\documentclass[12pt,a4paper,oneside,reqno]{amsart} 
\usepackage[a4paper,hmargin={2.4cm,2.4cm},vmargin={3cm,3cm}]{geometry}
\usepackage{amsmath}
\usepackage{amsfonts}
\usepackage{amssymb}
\usepackage{amsthm}
\usepackage{enumerate}
\usepackage{color}
\usepackage{url}
\usepackage{graphicx}
\usepackage{hyperref}

\theoremstyle{plain}
\newtheorem*{theorem*}{Theorem}
\newtheorem{theorem}{Theorem}[section]
\newtheorem{lemma}[theorem]{Lemma}
\newtheorem{proposition}[theorem]{Proposition}

\newtheorem{assumption}{Assumption}

\newtheorem*{acknowledgment}{Acknowledgment}

\theoremstyle{remark}

\theoremstyle{definition}

\numberwithin{equation}{section}
\newcommand{\thmref}[1]{Theorem~\ref{#1}}
\newcommand{\secref}[1]{Section~\ref{#1}}
\newcommand{\lemref}[1]{Lemma~\ref{#1}}
\newcommand{\propref}[1]{Proposition~\ref{#1}}

\newcommand{\Ind}[1]{\mathbf{1}_{\left\{#1\right\}}}

\newcommand{\PP}{\mathbb{P}}
\newcommand{\NN}{\mathbb{N}}
\newcommand{\EE}{\mathbb{E}}

\newcommand{\ZZ}{\mathbb{Z}}
\newcommand{\RR}{\mathbb{R}}
\newcommand{\cA}{\mathcal{A}}
\newcommand{\cB}{\mathcal{B}}

\newcommand{\cE}{\mathcal{E}}
\newcommand{\cF}{\mathcal{F}}
\newcommand{\cG}{\mathcal{G}}
\newcommand{\cI}{\mathcal{I}}

\newcommand{\cP}{\mathcal{P}}

\newcommand{\cS}{\mathcal{S}}

\newcommand{\cX}{\mathcal{X}}

\newcommand{\Var}{\operatorname{Var}}
\newcommand{\floor}[1]{\left\lfloor #1 \right\rfloor}

\newlength{\dhatheight}

\begin{document} 

\title[]{Randomly trapped random walks on $\mathbb Z^d$}
\author[]{Ji\v{r}\'{i} \v{C}ern\'{y}}
\author[]{Tobias Wassmer}
\address{\tiny University of Vienna\\Faculty of
Mathematics\\Oskar-Morgenstern-Platz 1\\A-1090 Wien\\Austria}
\email{\tiny
\href{mailto:tobias.wassmer@univie.ac.at}{tobias.wassmer@univie.ac.at}}

\begin{abstract}
  We give a complete classification of scaling limits of randomly trapped
  random walks and associated clock processes on $\mathbb Z^d$, $d\ge 2$.
  Namely, under the hypothesis that the discrete skeleton of the
  randomly trapped random walk has a slowly varying return probability, we
  show that the scaling limit of its clock process is
  either deterministic linearly growing or a stable
  subordinator. In the case when the discrete skeleton is
  a simple random walk on $\mathbb Z^d$, this implies that the
  scaling  limit of the randomly trapped random walk is either
  Brownian motion or the Fractional Kinetics process, as conjectured in
  \cite{BCCR13}.
\end{abstract}

\date{\today}

\maketitle


\section{Introduction}

Randomly trapped random walks (RTRWs) were introduced in \cite{BCCR13} for
two main reasons. On one hand they generalize several classical models
of trapped random walks such as the continuous-time random walk or the
symmetric Bouchaud trap model. On the other hand they provide a tool
for describing random walks on some classical random structures such as
the incipient critical Galton Watson tree or the invasion percolation
cluster on a regular tree.

In \cite{BCCR13} the authors define the RTRW on general graphs and study
in depth the model on $\ZZ$. They give a complete classification of
scaling limits, showing that the limit of a RTRW on $\mathbb Z$ is one of
the following four processes: (i) Brownian motion, (ii) Fractional
Kinetics process, (iii) FIN singular diffusion, or (iv) a new class of
processes called spatially subordinated Brownian motion. They further
give sufficient conditions for convergence to the respective limits and
study in detail how the different limits arise. For RTRW on $\mathbb Z^d$,
$d\ge 2$, they conjectured that only the first two of the above scaling
limits are possible, that is RTRW on $\mathbb Z^d$ converges after
rescaling either to the Brownian motion or to the Fractional Kinetics
process. We prove this conjecture here.

Let us briefly introduce the model, its formal definition is given in
Section~\ref{sec:modelandresults} below. The RTRW on $\mathbb Z^d$ is a
particular class of random walk in random environment.  Its law is
determined by two inputs: (i) its step distribution, that is a
probability measure $\nu$ on $\mathbb Z^d$, and (ii) a probability
distribution $\mu $ on the space of all probability measures on
$(0,\infty)$ characterising its waiting times. The random environment of
the RTRW is given by an i.i.d.~collection
$\pi = (\pi_x)_{x\in \mathbb Z^d}$ of $\mu $-distributed probability
measures. For fixed $\pi$, the RTRW $X=(X(t))_{t\ge 0}$ is a
continuous-time process such that, whenever at vertex $x$, it stays there
a random duration sampled from the distribution $\pi_{x}$ and then moves
on according to the transition kernel $\nu(\,\cdot\, - x)$. If the
process $X$ visits $x$ again at a later time, the duration of this next
visit at $x$ is sampled again and independently from the distribution
$\pi_x$. We always assume that $X$ starts at $0\in \mathbb Z^d$ and use
$\mathbb P$ for the annealed distribution of the process $X$.

From the description above it is apparent that the RTRW is a time change
of the discrete-time random walk $(Y(n))_{n\ge 0}$ on $\mathbb Z^d$ with
one-step distribution $\nu$. Formally, $X$ can be written as
\begin{equation}
  \label{e:protimechange}
  X(t) = Y(S^{-1}(t)),
\end{equation}
where the time-change process $S:\mathbb N\to[0,\infty)$, the \emph{clock
  process}, measures the time needed for a given number of steps of the
RTRW and $S^{-1}$ is its right-continuous inverse.  In view of
\eqref{e:protimechange} it should not be surprising that the scaling
behaviour of $X$ is (essentially) determined by the scaling behaviour of
the clock process.

While we are primarily interested in $Y$ being a simple random walk on
$\mathbb Z^d$, $d\ge 2$, it does not complicate the proofs to make
the following far less restrictive assumption on the random walk $Y$, that is on the
one-step distribution $\nu $: Let
$r_n: \NN\to[0,1]$ be the probability that $Y$ does not return to its
starting point in $n$ steps,
\begin{equation*}
  r_n = \mathbb P[Y(k)\neq Y(0) \text{ for }
    k=1,\dots,n].
\end{equation*}
\begin{assumption}
  \label{ass:slowvar}
  The function $r_n$ can be written as $r_n = \frac{1}{\ell^*(n)}$ for a slowly
  varying function $\ell^{*}:\mathbb N \to [1,\infty)$.
\end{assumption}

Assumption~\ref{ass:slowvar} is obviously fulfilled for all transient
random walks, where $1/\ell^*(n)\to \gamma$ for some $\gamma\in(0,1)$,
but there are also recurrent walks satisfying it with $\ell^*(n)\to\infty$
as $n\to\infty$. In particular, the classical result of Kesten and Spitzer
\cite[Theorem~3]{KS63} implies that this assumption holds for all
random walks on $d\ge 2$ for which the subgroup of $\ZZ^d$
generated by the set $\{x:\nu(x)>0\}$ is $d$-dimensional
(we will call this ``genuinely $d$-dimensional'').

We can state our first main theorem giving
the complete classification of the scaling limits of the clock
process.

\begin{theorem}
  \label{thm:general}
  Let $S:\mathbb N\to [0,\infty)$ be the clock process of the RTRW.
  Suppose that Assumption~\ref{ass:slowvar} holds and there is a
  sequence $a_N\nearrow \infty$ such that for all but countably many
  $t\in [0,\infty)$
  \begin{equation}
    \label{e:SN}
    S_N(t) := \frac{1}{a_N} S(\floor{Nt})\xrightarrow{N\to\infty}\mathcal
    S(t) \qquad \text{in $\mathbb P$-distribution},
  \end{equation}
  where $\mathcal S:[0,\infty)\to[0,\infty)$ is a cadlag process
  satisfying the non-triviality assumption
  \begin{equation}
    \label{e:Snontriv}
    \limsup_{t\to\infty} \mathcal S(t)  = \infty \qquad \mathbb
    P\text{-a.s.}
  \end{equation}
  Then one of the following two cases occurs:
  \begin{enumerate}[(i)]
    \item The limit clock process is linear, $\cS(t)=Mt$ for some
    constant $M>0$, and the normalizing sequence satisfies $a_N = N \ell(N)$
    for some slowly varying function $\ell$.
    \item The limit clock process is an $\alpha$-stable subordinator,
    $\cS=V_\alpha$, $\alpha \in (0,1)$, and the normalizing sequence satisfies
    $a_N = N^{1/\alpha} \ell(N) $ for some slowly varying function $\ell$.
  \end{enumerate}
\end{theorem}

In order to study the scaling limits of the RTRW itself, we need a more
restrictive assumption:
\begin{assumption}
  \label{ass:BM}
  The random walk $Y$ is genuinely $d$-dimensional.
  Its one-step distribution $\nu$ is centred,
	$\mathbb E[Y(1)]=0$, and has finite range, $\mathbb P[|Y(1)|>C]=0$ for
	some $C<\infty$.
\end{assumption}

This assumption ensures that the scaling limit of $Y$
is a $d$-dimensional Brownian motion: There exists a $d\times d$ matrix
$\mathcal A$ such that
\begin{equation}
  \label{e:bmlimit}
  Y_N(t):= \frac 1{\sqrt N}\, \mathcal A Y(\floor{Nt})
\end{equation}
converges to a standard $d$-dimensional Brownian motion. Note that by
the remark after Assumption~\ref{ass:slowvar}, for $d\geq2$
Assumption~\ref{ass:slowvar} is implied by Assumption~\ref{ass:BM}.

Our second main result classifies the possible scaling limits of RTRW and
confirms the conjecture of \cite{BCCR13}.

\begin{theorem}
  \label{thm:class}
  Let $d\geq2$ and $X:[0,\infty)\to \mathbb Z^d$ be the RTRW.
  Suppose that Assumption~\ref{ass:BM} holds and there is a
  sequence $a_N\nearrow\infty$ such that the processes
  \begin{equation}
    \label{e:XN}
    X_N(t) := \frac{1}{\sqrt N}\, \mathcal A X(a_N t) = Y_N(S_N^{-1}(t))
  \end{equation}
  converge
  in $\PP$-distribution on the space $D^d$ of cadlag
  $\mathbb R^d$-valued functions equipped with the
  Skorohod $J_1$-topology to some process
  $\cX:[0,\infty)\to \mathbb R^d$ satisfying the non-triviality assumption
  \begin{equation}
    \label{e:Xnontriv}
    \limsup_{t\to\infty} | \mathcal X(t)| = \infty \qquad \mathbb
    P\text{-a.s.}
  \end{equation}
  Then one of the following two cases occurs:
  \begin{enumerate}[(i)]
    \item $a_N = N \ell(N)$ and $\cX(t) = B(M^{-1}t)$ for some constant
    $M>0$, some slowly varying function $\ell$, and a standard $d$-dimensional
    Brownian motion $B$.
    \item $a_N = N^{1/\alpha} \ell(N) $ for some slowly varying
    function $\ell$ and a parameter $\alpha \in (0,1)$, and $\cX(t)
    = B(V_{\alpha}^{-1}(t))$, where $B$ is a standard $d$-dimensional
    Brownian motion and $V_{\alpha}^{-1}(t) = \inf\{s\geq0:~V_{\alpha}(s) >
      t\}$ is the right-continuous inverse of an $\alpha$-stable subordinator
    $V_{\alpha}$ which is independent of $B$ (i.e.~$\cX$ is the Fractional
      Kinetics process).
  \end{enumerate}
\end{theorem}

Let us make a few remarks about our setting and
results. The definition of the RTRW we give here is slightly more general
than the one in \cite{BCCR13} since we allow the discrete skeleton to be
more general than the simple random walk only. Assumption~\ref{ass:slowvar}
on the discrete skeleton is taken from \cite{FM}. This assumption can be
used to show weak laws of large numbers for the range of the random walk
and for some related quantities. We would like to point out that the only
place in the proof of Theorem~\ref{thm:general}
where we use $\ZZ^d$-specific properties of the random walk is in
the derivation of these laws of large numbers. In particular,
\thmref{thm:general} classifying the possible scaling limits of the clock
process can be shown to hold for the RTRW on any countable state space
where the discrete-time skeleton is a Markov chain satisfying such laws
of large numbers for the range and the related quantities.

Our setting generalizes several previous results, let us mention some of
them. Mostly, the models studied in the literature involve trapped random
walks with some kind of heavy-tailed waiting times, with the aim to show
convergence of rescaled clock processes to an $\alpha$-stable subordinator.

In the so-called continuous-time random walk (CTRW), introduced in
\cite{MW65}, all $\pi_x$ are deterministically identical heavy-tailed
probability distributions, that is for some $\alpha\in(0,1)$ and $c>0$, and
some slowly varying function $\ell$,
\begin{equation} \label{e:ctrw}
	\pi_x[u,\infty) = u^{-\alpha} \ell(u) \text{ as } u\to\infty.
\end{equation}
Independently of the nature of the discrete skeleton $Y$, the clock
process is then a sum of i.i.d.~heavy-tailed random variables, and it is
well known that it converges after normalization to a stable subordinator.
The scaling limits of the CTRW were studied in more detail in \cite{MS04}.

In the symmetric Bouchaud trap model (BTM) the discrete skeleton $Y$ is
simple random walk and the $\pi_x$ are exponential random variables with
means $m_x$ that are i.i.d.~heavy-tailed random variables satisfying e.g.
\begin{equation} \label{e:btm1}
	\PP[m_x>u] = c u^{-\alpha} (1+o(1)) \text{ as } u\to\infty,
\end{equation}

The BTM on $\ZZ^2$  was for the first time studied  in \cite{BCM06} where
the authors show convergence of the clock process to a stable
subordinator and use this to derive aging properties of the model. In
\cite{BC07} it is then shown, in the case of the BTM on $\mathbb Z^d$,
$d\ge 2$, that the rescaled random walks and clock processes converge
jointly to a Brownian motion and a stable subordinator, and therefore the
scaling limit of the BTM is the Fractional Kinetics process.

A general model of trapped random walk where the waiting times are
exponential with heavy-tailed means as in \eqref{e:btm1} is studied in
\cite{FM}. As mentioned above, they consider the discrete skeleton to be
an arbitrary random walk on $\ZZ^d$ satisfying
Assumption~\ref{ass:slowvar}. Instead of scaling limits, which require
additional restrictions as in our Assumption~\ref{ass:BM}, \cite{FM}
focus on the so-called age process, which is related to the clock process
and describes the `depth of the trap in which the process stays at a
given time'.

Our setting is restricted to the fact that the discrete skeleton $Y$ is
independent of the random environment $\pi $. There are however
interesting models where this is not the case, for example the asymmetric
Bouchaud trap model (ABTM). In \cite{BC11} for $d\ge3$ and in \cite{M11}
with different methods for $d\ge5$ it is shown that the scaling limit for
ABTM is also Fractional Kinetics. Yet another approach to prove
convergence of rescaled clock processes to a stable subordinator is given
in \cite{GS13}, their setting includes the ABTM as a special case.


The majority of the above mentioned previous results are quenched, that is
the convergence holds for almost every realisation of the environment. On
the contrary, our results are annealed, that is averaged over the
environment, but this is not an issue for the classification theorem.

We also believe that when the annealed convergence takes place as in
Theorem~\ref{thm:general}, then the quenched convergence holds true as
well. In high dimensions ($d\ge 5$) this could be proved  similarly as in
\cite{M11}, using techniques from \cite{BS02}, see also the additional
condition in \cite{FM} under which their annealed result holds quenched.
In low dimensions these methods fail due to many self-intersections
of the discrete skeleton. An adaptation of more complicated
methods which give the quenched convergence in low dimensions
(like the coarse-graining procedure of \cite{BCM06,BC07} or
the techniques of \cite{GS13}) to the RTRW seems to be non-trivial and is
out of the scope of this paper.

\medskip

We conclude the introduction by giving sufficient conditions for convergence in
both cases of our main theorems. Given the collection of probability
measures $\pi=(\pi_x)_{x\in\ZZ^d}$, let
$m_x=\int u \pi_x(du)\in(0,\infty]$ be the mean and
$\hat{\pi}_x(\lambda)=\int e^{-\lambda u} \pi_x(du)$ the Laplace transform
of $\pi_x$.
Note that in the next theorem Assumption~\ref{ass:slowvar} is not needed,
we only need $Y$ to be non-degenerate.

\begin{theorem} \label{thm:brown}
  Let $X$ be RTRW in $d\geq1$. If $\nu \neq \delta_0$ and the annealed expected waiting time is
  finite, $\EE[m_0] =M <\infty$, then the rescaled clock processes
  $S_N$ with normalization $a_N=N$ converge in $\PP$-distribution on
  $D^1$ equipped with the Skorohod $J_1$-topology to the linear process $\cS(t)=Mt$. If in addition
  Assumption~\ref{ass:BM} holds, then the rescaled processes $X_N$ with
  $a_N=N$ converge in $\PP$-distribution on $D^d$ equipped with the Skorohod $J_1$-topology, and the limit is
  $\cX(t)=B(M^{-1}t)$ as in $(i)$ of \thmref{thm:class}.
\end{theorem}

For convergence to Fractional Kinetics we have the following sufficient
criterion. In \secref{sec:proofsuff} we will sketch some examples of RTRWs
that satisfy this criterion with different functions $f$.

\begin{theorem}
  \label{thm:fk}
  Let $X$ be RTRW with discrete skeleton $Y$ satisfying
  Assumption~\ref{ass:slowvar} with slowly varying function $\ell^{*}$.
  Assume that there is a normalizing sequence $a_N\nearrow\infty$
  such that for any positive real number $r>0$ and a continuous
	function $f$,
  \begin{equation} \label{cond:fk}
    - \log \EE\left[\hat{\pi}_0(\lambda/a_N)^{r\ell^{*}(N)}\right] =
    f(r)\lambda^{\alpha}\frac{\ell^{*}(N)}{N} (1+o(1)) \text{ as }
    N\to\infty.
  \end{equation}
  Then the rescaled clock processes $S_N$ with normalization $a_N$
  converge in $\PP$-distribution on $D^1$ equipped with the Skorohod $M_1$-topology
  to an $\alpha$-stable subordinator $V_{\alpha}$. If in addition
  Assumption~\ref{ass:BM} holds, then the rescaled processes $X_N$
  converge in $\PP$-distribution on $D^1$ equipped with the Skorohod $J_1$-topology,
  and the limit is the FK process as in $(ii)$ of \thmref{thm:class}.
\end{theorem}

The rest of this paper is structured as follows. In
\secref{sec:modelandresults} we give precise definitions of the model and
introduce some notation used through the paper. \thmref{thm:general} and
\thmref{thm:class}  are proved in Sections~\ref{sec:proofgen}
and~\ref{sec:proofz} respectively, and Section~\ref{sec:proofsuff} deals
with Theorems~\ref{thm:brown} and \ref{thm:fk}. Finally, in
\secref{sec:proofs} we prove one technical lemma which is used in the
proof of Theorem~\ref{thm:general}. In Appendix~\ref{appendix} we explain
how Assumption~\ref{ass:slowvar} on the escape probability implies the
laws of large numbers that we mentioned above.

\medskip

\begin{acknowledgment}
The authors would like to thank the referee for carefully reading the
manuscript and giving important comments that helped to improve the paper.
\end{acknowledgment}


\section{Setting and notations}
\label{sec:modelandresults}

We start by giving a formal definition of the RTRW.
Recall that $\nu $ is a probability measure on $\mathbb Z^d$ and $\mu $ a
probability measure on the space of probability measures on $(0,\infty)$.
To avoid trivial situations, we assume that $\nu \neq \delta_0$.

Given $\mu $ and $\nu $, let $\pi = (\pi_x)_{x\in \mathbb Z^d}$ be an
i.i.d.~sequence of probability measures with marginal $\mu $, and
$\xi = (\xi_i)_{i\ge 1}$ an i.i.d.~sequence with marginal $\nu $ independent of
$\pi $ defined on some probability space
$(\Omega , \mathcal F, \mathbb P)$.
Define
\begin{equation*}
  Y(n)=\xi_1+\dots+\xi_n
\end{equation*}
to be a random walk with step distribution $\nu $ and denote by
$L(x,n)=\sum_{k=0}^n \Ind{Y(k)=x}$ its local time.

Given a realisation of
$\pi $, let further $(\tau_x^i)_{x\in \mathbb Z^d,i\ge 1}$ be a collection
of independent random variables, independent of $\xi $, such that every
$\tau_x^i$ has distribution $\pi_x$, defined on the same probability
space. The clock process of the RTRW, $S:\mathbb N\to[0,\infty)$ is then
defined by $S(0)=0$ and
\begin{equation*}
  S(n) = \sum_{x\in\ZZ^d}
  \sum_{i=1}^{L(x,n-1)} \tau_x^i = \sum_{k=0}^{n-1} \tau_{Y(k)}^{L(Y(k),k)}
  \quad \text{ for }n\geq1.
\end{equation*}
Finally, we define the RTRW $X=(X(t))_{t\ge 0}$ by
\begin{equation*}
  X(t) = Y(k) \qquad \text{ for } S(k) \leq t < S(k+1),
\end{equation*}
or equivalently
\begin{equation*}
  X(t) = Y(S^{-1}(t)),
\end{equation*}
where $S^{-1}(t) = \inf\{k\geq0:~S(k) > t\}$ is the right-continuous
inverse of $S$.

Under $\mathbb P$, the process $X$ has exactly the law described in the
introduction. The random variable $\tau_x^i$ denotes the duration of the
$i$-th visit of the vertex $x$.
We refer to $\mathbb P$ as
\emph{annealed} distribution of $X$.

We write $D^d$ for the space of the $\mathbb R^d$-valued cadlag functions on
$[0,\infty)$, and when needed $D^d(J_1)$, $D^d(M_1)$, $D^d(M_1')$ to
point out which of Skorohod topologies we use on this space. We refer to
\cite[Chapter~3.3]{W02} for an introduction and \cite[Chapters~12--13]{W02}
for details on these topologies. The less usual $M_1'$-topology, which plays a role only
in \propref{prop:convclock}, is a modification of the Skorohod $M_1$-topology
which is convenient for dealing  with irregularities at the origin, see
\cite[Section~13.6.2]{W02}. In any case, we will never need to know
the actual definitions of these topologies, we only use them when
applying results from \cite{W02}.

It will be useful to introduce the sequence of successive waiting times
\begin{equation*}
  \tilde \tau_k = \tau_{Y(k)}^{L(Y(k),k)}, \qquad k\ge 0.
\end{equation*}
With this notation,
\begin{equation}
  \label{e:Stildetau}
  S(n)=\sum_{k=0}^{n-1} \tilde \tau_k.
\end{equation}

We now show that $\tilde \tau_k$ is ergodic, which will be used in the
proof of \thmref{thm:brown}. To this end let
$\mathcal P'$ be the law on $\Omega ':=[0,\infty)^{\mathbb N}$ of the sequence
$(\tilde \tau_k)_{k\ge 0}$ and let $\theta $ be the left shift on
$\Omega '$,  $\theta (\tilde \tau_1, \tilde \tau_2,\dots) = (\tilde\tau_2,\tilde \tau_3,\dots)$.

\begin{lemma}
  \label{lem:ergod}
  The left-shift $\theta$ acts ergodically
  on $(\Omega',\cP')$.
\end{lemma}

\begin{proof}
  To show that $\theta $ is measure-preserving we follow
  the environment as `viewed from the particle'. Namely, let
  $\Theta:\Omega \to \Omega $ be such that
  if $\omega' = \Theta(\omega )$, then
  \begin{equation*}
    \begin{split}
      &\xi_i(\omega')=\xi_{i+1}(\omega ), \qquad i\ge 1,
      \\&\pi_x(\omega')=\pi_{x+\xi_1(\omega )}(\omega ), \qquad x\in \mathbb
      Z^d,
      \\&\tau_x^i(\omega') = \begin{cases}
        \tau_{x+\xi_1(\omega )}^i,\qquad & \text{if }x\neq -\xi_1(\omega),
        \\\tau_{x+\xi_1(\omega )}^{i+1},&\text{if }x=-\xi_1(\omega ).
      \end{cases}
    \end{split}
  \end{equation*}
  From the independence of $\xi $ from $\pi $ and $\tau $, and from the
  i.i.d.~properties of $\pi $ and $\tau_x^\cdot$ for every $x$, it is easy
  to see that the law of $X\circ\Theta$ agrees with the law of $X$,
  that is $\Theta$ is $\mathbb P$-preserving. Since, in addition,
  $\tilde \tau (\Theta (\omega ))= \theta (\tilde \tau (\omega ))$ and
  $\mathcal P' = \mathbb P \circ \tilde \tau^{-1}$, this implies that
  $\theta$ is $\mathcal P'$-preserving.

  To prove the ergodicity, we show that $\theta $ is strongly mixing. To
  this end it is sufficient to verify that
  \begin{equation}
    \label{def:mixergodic}
    \left|\cP'[\theta ^{-n}A\cap B] - \cP'[A]\cP'[B]\right|
			\xrightarrow{n\to\infty} 0
  \end{equation}
  for all cylinder sets $A=\{\tilde{\tau}_i \in A_i,~i\in I\}$,
  $B=\{\tilde{\tau}_j \in B_j,~j\in J\}$, where $I,J\subset\NN$ are
  finite sets and $A_i,B_j\subset\RR$ are Borel sets, see
  e.g.~\cite[Prop 2.5.3]{P83}. Fix two such sets $A$ and $B$ and define
	the event $\cI_n(A,B) = \{Y(i+n) = Y(j) \text{ for some } i\in I,
  j\in J\}$. Denote by $\cG_k^m=\sigma(\xi_{k+1},\xi_{k+2},\dots,\xi_m)$
	the $\sigma$-algebra generated by the steps made by the random walk
	between time $k$ and $m$, and write
	$\cG_{(n)}=\cG_{0}^{(\max I+n)\vee \max J}$,
	$\cG_{(I,n)}=\cG_{\min I + n}^{\max I+n}$,
	$\cG_{(J)}=\cG_{\min J}^{\max J}$.
	By the independence structure of the $\tau_x^i$ we have that
	\begin{equation} \label{eq:indAB}
		\cP'\left[\theta^{-n}A\cap B\mid \cG_{(n)}\right] \mathbf{1}_{\cI_n(A,B)^c}
		= \cP'\left[\theta^{-n}A\mid \cG_{(n)}\right]
			\cP'\left[B\mid \cG_{(n)}\right] \mathbf{1}_{\cI_n(A,B)^c}.
	\end{equation}
	Moreover,
	\begin{equation}\label{eq:depABonY}\begin{split}
		\cP'\left[\theta^{-n}A\mid \cG_{(n)}\right]
			&= \cP'\left[\theta^{-n}A\mid \cG_{(I,n)}\right], \\
		\cP'\left[B\mid \cG_{(n)}\right]
			&= \cP'\left[B\mid \cG_{(J)}\right].
	\end{split}\end{equation}
	By the independence of the $\xi_k$, as soon as $\max J<\min I + n$ the
	right hand sides of the above two equations are independent. Denote by
	$\cE'$ the 	expectation corresponding to $\cP'$. 	Using \eqref{eq:indAB},
	\eqref{eq:depABonY} and the independence of the two right hand sides in
	\eqref{eq:depABonY}, and the fact that $\cP'[ \theta^{-n}A]=\cP'[A]$, we
	have for $n$ large enough,
	\begin{equation} \label{eq:mixmix}\begin{split}
		\cP'\left[\theta^{-n}A\cap B\right]
			&= \cE'\left[\cP'\left[\theta^{-n}A\cap B\mid \cG_{(n)}\right]
				\left(\mathbf{1}_{\cI_n(A,B)}+ \mathbf{1}_{\cI_n(A,B)^c}\right)
					\right]\\
			&= \cE'\left[\cP'\left[\theta^{-n}A\mid \cG_{(n)}	\right]
				\cP'\left[B\mid \cG_{(n)}\right] \mathbf{1}_{\cI_n(A,B)^c}\right]
				+ O\left(\cP'[\cI_n(A,B)]\right)\\
			&= \cE'\left[\cP'\left[\theta^{-n}A\mid \cG_{(I,n)}	\right]
				\cP'\left[B\mid \cG_{(J)}\right]\right]
					+ O\left(\cP'[\cI_n(A,B)]\right) \\
			&= \cP'[\theta^{-n}A]\cP'[B]+ O\left(\cP'[\cI_n(A,B)]\right) \\
			&= \cP'[A]\cP'[B]+ O\left(\cP'[\cI_n(A,B)]\right)
	\end{split}\end{equation}
	But
	\begin{equation*}
		\cP'[\cI_n(A,B)] \leq \sum_{i\in I,j\in J} \PP[Y(i+n) = Y(j)],
	\end{equation*}
  and for $n$ large enough the Markov property for $Y$ implies that
  $\PP[Y(i+n) = Y(j)]=\PP[Y(i+n-j) = 0]$, which tends to $0$ as
  $n\to\infty$ for every (non-degenerate) random walk, see
  e.g.~\cite[P7.6]{S76}.
  Since $I$ and $J$ are finite, $\cP'[\cI_n(A,B)]\to0$ as $N\to\infty$,
	and thus \eqref{def:mixergodic} follows from \eqref{eq:mixmix}.
\end{proof}


\section{Proof of \thmref{thm:general}} 
\label{sec:proofgen}

In this section we prove \thmref{thm:general}.
In the next two lemmas we study the properties of the limit
clock process $\mathcal S$.

\begin{lemma}
  \label{lem:propclock}
	If the random walk $Y$ satisfies Assumption~\ref{ass:slowvar} and
	the rescaled clock processes $S_N$ converge to $\mathcal S$ in
	the way as stated in \thmref{thm:general},
  then the limit clock process $\cS$ has stationary increments and is
  self-similar with index $\rho>0$, i.e.~$\cS(t) \stackrel{d}{=}
  \lambda^{\rho}\cS(t/\lambda)$. Moreover, the normalizing sequence is of
  the form $a_N= N^{\rho} \ell(N)$, for the same $\rho>0$ and some slowly
  varying function $\ell$.
\end{lemma}

\begin{proof}
  Stationarity of the increments follows immediately from
  \eqref{e:Stildetau} and the stationarity of the sequence $\tilde \tau $
  of successive waiting times  which was proved in Lemma~\ref{lem:ergod}.
  To see the self-similarity, fix $\lambda >0$ and $t$ such that
  condition \eqref{e:SN} holds for $t$ and
  $t/\lambda$, and $\mathcal S(t)$, $\mathcal S(\lambda t)$ are
  not identically zero, which is possible thanks to \eqref{e:Snontriv}.
  Then,
  \begin{equation*}
    \cS(t) = \lim_{N\to\infty} \frac{1}{a_N} S(Nt)
    = \lim_{N\to\infty} \frac{a_{\lambda N}}{a_N} \frac{1}{a_{\lambda
        N}} S\left(\lambda N \frac{t}{\lambda}\right) \stackrel{d}{=}
    \cS\left(\frac{t}{\lambda}\right) \lim_{N\to\infty} \frac{a_{\lambda
        N}}{a_N}.
  \end{equation*}
  Since $\mathcal S(t)$ and $\mathcal S(t/\lambda )$ are not identically
  zero, it follows that $\frac{a_{\lambda N}}{a_N}$ must
  converge to some constant $c(\lambda)$, yielding the scale invariance.
  Moreover, elementary results of the theory of regularly varying
  functions imply that $c(\lambda)=\lambda^{\rho}$ for some
  $\rho\in \RR$, and that $a_N$ is regularly varying of index $\rho$, that
  is $a_N=N^{\rho}\ell(N)$ for some slowly varying function $\ell(N)$.
  Note that $\rho>0$ since $\rho=0$ would imply
	$\lim_{N\to\infty} \frac{a_{\lambda N}}{a_N}=1$, hence
	$\cS(t)\stackrel{d}{=}\cS(t/\lambda)$, which violates the non-triviality
	assumption \eqref{e:Snontriv}.
\end{proof}

\begin{lemma}
  \label{lem:indincr}
  If the random walk $Y$ satisfies Assumption~\ref{ass:slowvar} and
	the rescaled clock processes $S_N$ converge to $\mathcal S$ in
	the way as stated in \thmref{thm:general},
	then the limit clock process $\cS$ has independent increments.
\end{lemma}

Let us postpone the proof of this lemma and show Theorem~\ref{thm:general}
first.

\begin{proof}[Proof of \thmref{thm:general}]
  By Lemmas~\ref{lem:propclock} and~\ref{lem:indincr}, $\cS$ has stationary
  and independent increments and is self-similar with index $\rho $.
  From this, the fact that $\cS\ge 0$
	and the non-triviality assumption \eqref{e:Snontriv}
	it follows that
  either $\rho=1$ and $\cS(t) = Mt$ for some $M\in (0,\infty)$, or
  $\rho>1$ and $\cS$ is an increasing $\alpha$-stable L\'evy process
  with $\alpha=\rho^{-1}\in(0,1)$,  that is an $\alpha $-stable
  subordinator. \lemref{lem:propclock} gives the normalizing sequence $a_N$
  as claimed.
\end{proof}

In order to show Lemma~\ref{lem:indincr} we need three technical lemmas
which are consequences of laws of large numbers for the range-like
objects related to the random walk $Y$, as mentioned in the introduction.

The first lemma states that for any given times $0=t_0<t_1<\cdots<t_n=t$,
the number of vertices visited by the random walk $Y$ in more than one of
the time intervals $[\floor{t_{i-1}N},\floor{t_i N}-1]$ is small. To this
end, let
\begin{equation*}
  R(k)=\{Y(0),\dots,Y(k-1)\}
\end{equation*}
be the range of the random
walk $Y$ at time $k-1$,  $R_N^i$ be the `range  between
$t_{i-1}N$ and $t_i N$',
\begin{equation*}
  R_N^i =
  \left\{Y(k):~k=\floor{Nt_{i-1}},\dots,\floor{Nt_i}-1\} \right\},
\end{equation*}
$O_N^i$ be the set of the points visited only in this time interval,
\begin{equation*}
  O_N^i = \left\{ x\in R_N^i:~x\notin R_N^j \text{ for all }j\neq i\right\},
\end{equation*}
and $M_N^i$ be the set of points visited in more than one of them,
$M_N^i = R_N^i\setminus O_N^i$.

\begin{lemma}
  \label{lem:llnmultincr}
  If $Y$ verifies Assumption~\ref{ass:slowvar},
  then for any choice of time points
  $0=t_0<t_1<\dots<t_n=t$,
  \begin{equation*}
    \lim_{N\to\infty} |M_N^i|\frac{\ell^*(N)}{N} = 0 \qquad
    \text{in } \mathbb P\text{-probability for all }i=1,\dots,n.
  \end{equation*}
\end{lemma}

\begin{proof}
  The size of the sets  $O^i_N$ can be bounded by
  \begin{equation}
    |R(\floor{Nt})|- \left| \bigcup_{j\neq i} R_N^j \right|
    \leq |O_N^i|
    \leq \left| R(\floor{Nt_i}) \right| -
      \left|R(\floor{Nt_{i-1}})\right|. \label{eq:llnlem2}
  \end{equation}
  Applying the laws of large numbers from Lemma~\ref{lem:llnrangegeneral} and the Markov property at times
  $\floor{Nt_i}$, it follows that for every $i=1,\dots,n$,
  \begin{equation*}
      |R(\floor{Nt_i})|\frac{\ell^*(N)}{N t_i} \xrightarrow{N\to\infty}1,
      \quad\text{and}\quad
      \left| \bigcup_{j\neq i} R_N^j \right|\frac{\ell^*(N)}{N
        (t_n-t_i+t_{i-1})}\xrightarrow{N\to\infty}1
  \end{equation*}
  in probability. Inserting this into \eqref{eq:llnlem2} yields a law of
  large numbers for $|O_N^i|$,
  \begin{equation*}
    |O_N^i|\frac{\ell^*(N)}{N (t_i-t_{i-1})}\xrightarrow{N\to\infty}1
  \end{equation*}
  in probability. By Lemma~\ref{lem:llnrangegeneral} and the Markov property
  again, $|R_N^i|$ satisfies the same law of large numbers as $|O_N^i|$.
  Using $|M_N^i|=|R_N^i|-|O_N^i|$ the claim follows.
\end{proof}
The second lemma will help to control the contribution of frequently visited vertices to
the clock process. Fix $t>0$,
and for $K>0$ define the  set of `frequently visited vertices'
\begin{equation}
  \label{e:calFNK}
  \mathcal F_{N,K}=\big\{x:L(x,\floor{Nt}-1) \geq K \ell^*(N)\big\}.
\end{equation}
Let $F_{N,K}$ be the `number of visits to $\mathcal F_{N,K}$'
\begin{equation}
  \label{e:FNK}
  F_{N,K} = \sum_{x\in\mathcal F_{N,K}} L(x,\floor{Nt}-1).
\end{equation}

\begin{lemma}
  \label{lem:largelocaltime}
  If $Y$ verifies Assumption~\ref{ass:slowvar}, then there
  is a constant $c>0$ such that for every $\varepsilon >0$
	and fixed $t>0$
  \begin{equation*}
    \mathbb P[F_{N,K}\geq \epsilon  Nt] \leq
    \epsilon \qquad
    \text{for all $N$ large enough,}
  \end{equation*}
  with
  \begin{equation}
    K=K(\epsilon) = -c\log\left(\epsilon^2\right). \label{def:K}
  \end{equation}
\end{lemma}

\begin{proof}
  We claim that for $\epsilon$ small enough and $N$ large enough,
  \begin{equation}
    \label{eq:expecF}
    \mathbb E\left[F_{N,K}\right]
    \leq \epsilon^2 Nt.
  \end{equation}
  Applying the Markov inequality then yields the desired result.

  To show \eqref{eq:expecF}, let
  $\psi_k=\Ind{Y(l)\neq Y(k)\,\forall l<k}$ be the indicator of the
  event that a `new' vertex is found at time $k$.  Then
  \begin{equation*}
    F_{N,K} = \sum_{k=0}^{\floor{Nt}-1}\psi_k
    L(Y_k,\floor{Nt}-1)\Ind{L(Y_k,\floor{Nt}-1)\ge K\ell^*(N)}.
  \end{equation*}
  Using the Markov property and the fact that $L(Y_k,\floor{Nt}-1)$ is
  stochastically dominated by $L(0,\floor{Nt}-1)$,
  \begin{equation}
    \label{e:poi}
    \mathbb E [F_{N,K}] \le
    \mathbb E\big[L(0,\floor{Nt}-1)
      \Ind{L(0,\floor{Nt}-1)\ge K\ell^*(N)}\big]
    \sum_{k=0}^{\floor{Nt}-1}\mathbb E[\psi_k].
  \end{equation}
  By \eqref{e:ER},
  $\sum_{k=0}^{\floor{Nt}-1}\mathbb E[\psi_k]=\mathbb E [|R(\floor{Nt})|]
	=Nt/\ell^*(\floor{Nt})(1+o(1))$.
  On the other hand, denoting by $\tilde H_0$ the first return time of
  $Y$ to
  $0$, for every $k\ge 1$
  \begin{equation*}
    \mathbb P[L(0,\floor{Nt}-1)\ge k]
    \le \big(\mathbb P[\tilde H_0\le \floor{Nt}]\big)^{k-1}
    = (1-\ell^*(\floor{Nt})^{-1})^{k-1},
  \end{equation*}
  and thus $L(0,\floor{Nt}-1)$ is stochastically dominated by a geometric
  random variable with parameter $1/\ell^*(\floor{Nt})$. If $G$ is a
  geometric variable with parameter $p$, then for every $M\in \mathbb N$,
  \begin{equation*}
    \mathbb E[G\Ind{G\ge M}]= (1-p)^{M-1}\Big(M -1 +\frac 1p\Big).
  \end{equation*}
  Hence,
  \begin{equation}
    \begin{split}
      \label{e:poii}
      \mathbb E&\big[L(0,\floor{Nt}-1)
        \Ind{L(0,\floor{Nt}-1)\ge K\ell^*(N)}\big]
      \\&\le
      \left(1-\frac 1 {\ell^* (\floor{Nt})}\right)^{K\ell^*(N)-1}
      \big(K\ell^*(N) -1 + \ell^*(\floor{Nt})\big)
    \end{split}
  \end{equation}
  and the claim \eqref{eq:expecF} follows by inserting $K$ as in
  \eqref{def:K}, using the slow variation of $\ell^{*}$ and combining
	\eqref{e:poi}, \eqref{e:poii}.
\end{proof}

The last of the technical lemmas allows to control the influence of an
arbitrary subset of waiting times to the sum of all waiting times if the
subset is small.

\begin{lemma}
  \label{lem:badtimes}
  Let $\cB_N\subset\{0,1,\dots,\floor{Nt}-1\}$ be a
  random set, depending on the trajectory of the random walk $Y$ up to
  time $\floor{Nt}-1$ only. If
  Assumption~\ref{ass:slowvar} holds, then for every $t>0$ and $\delta>0$,
  \begin{equation*}
    \lim_{\epsilon\to0} \lim_{N\to\infty}
    \PP\left[\sum_{k\in \cB} \tilde{\tau}_k \geq \delta
      S(\floor{Nt}),~|\cB|\leq \epsilon N \right] =0.
  \end{equation*}
\end{lemma}

The proof of this lemma is surprisingly lengthy and is therefore
postponed to Section~\ref{sec:proofs}. The main source of complications
comes from the fact that we cannot make any assumptions on the moments of the
waiting times $\tau_x^i$. It is also essential to use some properties of
the random walk $Y$, as it is easy to construct counterexamples to the
lemma when $\tau_x^i$ are not summed along the trajectory of $Y$.

With the above three lemmas we can now show Lemma~\ref{lem:indincr}.

\begin{proof}[Proof of Lemma~\ref{lem:indincr}]
  Fix times $0=t_0<t_1<\dots<t_n=t $. Consider first
  the following alternative construction of the clock
  process $S$. On the same space $(\Omega,\cF,\PP)$, let for every
  $x\in\ZZ^d$ independently $(\pi_{x,j})_{j=1,\dots,n}$ be
  i.i.d.~$\mu$-distributed probability measures,
  and given a realisation of these measures, let $(\tau_{x,j}^i)_{x\in\ZZ^d, j,i\geq1}$
  be independent random variables such that every $\tau_{x,j}^i$ has
  distribution $\pi_{x,j}$. For every vertex $x\in\ZZ^d$, let
  $j(x)$ be be such that the first visit to $x$ occurs in the time
  interval $[\floor{Nt_{j(x)-1}},\floor{Nt_{j(x)}}-1]$.
  Define a new process $S':\mathbb N\to [0,\infty)$ by
  \begin{equation*}
    S'(k) = \sum_{x\in\ZZ^d} \sum_{i=1}^{L(x,k-1)} \tau_{x,j(x)}^i.
  \end{equation*}
  One can think of choosing the distributions $\pi_x$ at the time of the first
  visit in $x$ according to the time interval in which this first
  visit occurs. Constructed in this way, $S'$ has clearly the
  same distribution as the original clock process $S$.

  We now define an approximation $\tilde{S}$ of
  $S'$ which
  collects time $\tau_{x,j}^{L(x,k)}$ whenever at a vertex $x$
  at time $k\in[\floor{Nt_{j-1}},\floor{Nt_j}-1]$,
  \begin{equation} \label{def:stilde}
    \tilde{S}(m)
    = \sum_{j=1}^n \sum_{k=\floor{Nt_{j-1}}}^{(m\wedge\floor{Nt_j})-1} \tau_{Y(k),j}^{L(Y(k),k)}.
  \end{equation}
  $\tilde S$ can be viewed as the clock for which the whole environment
  $\pi$ is being refreshed at all times $\floor{Nt_j}$. Therefore, by the
  independence structure of the $\tau_{x,j}^i$'s, the increments
  $(\tilde{S}(\floor{Nt_j})-\tilde{S}(\floor{Nt_{j-1}}))_{j=1,\dots,n}$
  are mutually independent. In addition, for every $j$, the
  increment $\tilde{S}(\floor{Nt_j})-\tilde S(\floor{Nt_{j-1}})$ is
  independent of the increments
  $\{\xi_k:~k\notin[\floor{Nt_{j-1}},\floor{Nt_j}-1]\}$ of the random
  walk $Y$.

  To conclude the proof it is now sufficient to show that for all $j=1,\dots,n$ and every $\delta>0$,
  \begin{equation} \label{eq:stilde}
    \lim_{N\to\infty} \PP\left[ \left|\tilde{S}(\floor{Nt_j})-S'(\floor{Nt_j})\right| > \delta S'(\floor{Nt_j}) \right] =0.
  \end{equation}
  This implies that the limit process $\cS$ has independent increments.
  Indeed, note that \eqref{eq:stilde} readily implies
  $\frac{\tilde{S}(\floor{Nt_j})}{S'(\floor{Nt_j})}\to1$ in $\PP$-probability
  for all $j$. This means that whenever
  $\frac{1}{a_N}S'(\floor{Nt_j}) \stackrel{d}{\to} \cS(t_j)$, then also
  $\frac{1}{a_N}\tilde{S}(\floor{Nt_j}) \stackrel{d}{\to} \cS'(t_j)$, and
  therefore the increments $(\cS(t_j)-\cS(t_{j-1}))_{j=1,\dots,n}$ are
  independent, whenever \eqref{e:SN} is satisfied for the times $t_j$. By
  easy approximation arguments this also holds for the at most countably
  many $t_j$'s that do not satisfy \eqref{e:SN}. Since the times $t_j$
  are chosen arbitrarily,
  it follows that the process $\cS$ has independent increments.

  In order to show \eqref{eq:stilde}, note that the difference of
	$\tilde{S}(\floor{Nt_j})$ and $S'(\floor{Nt_j})$ originates in the waiting times in
  vertices visited in multiple time intervals.
  Recalling the sets $M_N^j$ from \lemref{lem:indincr},
  \begin{equation*}
    \left| \tilde{S}(\floor{Nt_j}) - S'(\floor{Nt_j})\right| \leq \sum_{l=1}^j \sum_{x\in M_N^l} \sum_{i=1}^{L(x,\floor{Nt_j}-1)} \tau_{x,l}^i.
  \end{equation*}
  It is therefore sufficient to show that for each $j=1,\dots,n$ and $1\leq l\leq j$, and every $\delta>0$,
  \begin{equation*}
    \lim_{N\to\infty} \PP\left[\sum_{x\in M_N^l}
      \sum_{i=1}^{L(x,\floor{Nt_j}-1)} \tau_{x,l}^i \geq \delta S'(\floor{Nt_j}) \right] =0.
  \end{equation*}
  The probability above is bounded by
  \begin{equation*}
    \PP\left[ (1+\delta) \left(\sum_{x\in M_N^l}\!\!\!\!\!\!\!\sum_{i=1}^{L(x,\floor{Nt_j}-1)} \tau_{x,l}^i \right) \geq
      \delta \left(\sum_{x\in R(\floor{Nt_j})\setminus M_N^l}\!\!\!\!\!\!\!\sum_{i=1}^{L(x,\floor{Nt_j}-1)} \tau_{x,j(x)}^i +
        \sum_{x\in M_N^l}\!\!\!\!\!\!\!\sum_{i=1}^{L(x,\floor{Nt_j}-1)} \tau_{x,l}^i \right) \right].
  \end{equation*}
  Note that, by definition of the random variables $\tau_{x,j}^i$,
  requiring the above probability to tend to $0$ as $N\to\infty$ is the same as requiring
  \begin{equation} \label{eq:ignoreMj}
    \lim_{N\to\infty} \PP\left[\sum_{x\in M_N^l}
      \sum_{i=1}^{L(x,\floor{Nt_j}-1)} \tau_{x}^i \geq \delta S(\floor{Nt_j}) \right] =0,
  \end{equation}
  for each $j=1,\dots,n$ and $1\leq l\leq j$, and every $\delta>0$, where here
  $S$ is the original clock process, i.e.~the sum of the $\tau_x^i$'s which have distributions $\pi_x$.

  Fix $\epsilon>0$ small, set $K$ as in \eqref{def:K}, recall the
  definition of $\mathcal F_{N,K}$ from \eqref{e:calFNK} (with $t_j$ instead of $t$),
  and write
  \begin{equation}
    \begin{split}
      \label{eq:indincr}
      &\PP \left[\sum_{x\in M_N^l} \sum_{i=1}^{L(x,\floor{Nt_j}-1)}
        \tau_x^i \geq \delta S(\floor{Nt_j}) \right]  \\
      &\leq \PP\left[\sum_{x\in M_N^l\setminus \mathcal F_{N,K}}
        \!\!\!\!\!\!\!\sum_{i=1}^{L(x,\floor{Nt_j}-1)} \tau_x^i \geq
        \frac{\delta}{2} S(\floor{Nt_j}) \right]
      + \PP\left[\sum_{x\in M_N^l\cap \mathcal F_{N,K}}
        \!\!\!\!\!\!\!\sum_{i=1}^{L(x,\floor{Nt_j}-1)} \tau_x^i \geq
        \frac{\delta}{2} S(\floor{Nt_j}) \right].
    \end{split}
  \end{equation}
  By \lemref{lem:llnmultincr} we can choose $N$ large enough such that
  $\PP[| M_N^l| > \epsilon N / \ell^*(N)] \leq \epsilon$. Then the first
  term on the right-hand side of \eqref{eq:indincr} is bounded by
  \begin{align*}
    & \PP\left[\sum_{x\in M_N^l\setminus \mathcal F_{N,K} }
      \sum_{i=1}^{L(x,\floor{Nt_j}-1)} \tau_x^i
      \geq \frac{\delta}{2}S(\floor{Nt_j}) ,~|M_N^l| \leq \epsilon N /
      \ell^*(N)\right] + \epsilon \\
    &= \PP\left[ \sum_{k\in\cB_1} \tilde{\tau}_k \geq \frac{\delta}{2}
      S(\floor{Nt_j}) ,~|M_N^l| \leq \epsilon N / \ell^*(N)\right]
    + \epsilon.
  \end{align*}
  Here $\cB_1$ is the set of all times where a vertex in $M_N^l\setminus\mathcal F_{N,K}$, i.e.~with
  $L(x,\floor{Nt_j}-1) \leq K\ell^*(N)$ is visited. But if
  $|M_N^l| \leq \epsilon N / \ell^*(N)$, then $|\cB_1|\leq \epsilon K N$.
  Since $\epsilon K \to 0$ as $\epsilon\to0$ by the definition of $K$, we
  can apply \lemref{lem:badtimes} to get that the first term on the
  right-hand side of \eqref{eq:indincr} converges to $0$ when
  $N\to \infty $ and then $\varepsilon \to 0$.

  The second term on the right-hand side of \eqref{eq:indincr}
  can be bounded similarly. Recalling $F_{N,K}$ (for $t_j$) from \eqref{e:FNK}, it is
  bounded from above by
  \begin{align*}
    &\PP\left[ F_{N,K} \geq \epsilon N\right] + \PP\left[
      \sum_{k\in\cB_F} \tilde{\tau}_k \geq \frac{\delta}{2} S(\floor{Nt_j}),
      F_{N,K} \leq \epsilon N \right].
  \end{align*}
  Here $\cB_F$ is the set of times where a frequently visited vertex is
  visited, i.e.~$|\cB_F|=F_{N,K}$. Applying \lemref{lem:largelocaltime}
  to the first term and \lemref{lem:badtimes} to the second, this
  converges to zero as $N\to\infty$ and $\varepsilon \to 0$,
  and \eqref{eq:ignoreMj} follows. This finishes the proof of the lemma.
\end{proof}


\section{Proof of \thmref{thm:class}} 
\label{sec:proofz}

The goal of this section is to prove the classification theorem for the
RTRW, Theorem~\ref{thm:class}. This will be done using
Theorem~\ref{thm:general}. At first we should however show that the
assumptions of Theorem~\ref{thm:class} allow to verify the hypotheses of
Theorem~\ref{thm:general}.

\begin{proposition}[$d\ge 1$]
  \label{prop:convclock}
  Let $X_N$ be as in \eqref{e:XN}. Suppose that Assumption~\ref{ass:BM}
  holds and that $X_N$ converge in the sense of Theorem~\ref{thm:class}.
  Then the clock processes $S_N$, defined as in \eqref{e:SN}, converge in
  $\PP$-distribution on $D^d(M_1')$ to some process $\cS$. If $\cS(0)=0$,
  then the convergence holds with respect to the Skorohod $M_1$-topology.
\end{proposition}

We first use this proposition to show \thmref{thm:class}.

\begin{proof}[Proof of \thmref{thm:class}]
  By \propref{prop:convclock}, $S_N$ converge to some process
  $\mathcal S$ in distribution on $D^d(M'_1)$. This convergence implies
  the convergence of $S_N(t)$ to $\mathcal S(t)$ for all but countably
  many $t\in [0,\infty)$, cf.~\cite[Theorem~11.6.6 and Corollary~12.2.1]{W02}.
	The non-triviality assumption~\eqref{e:Xnontriv}
  implies \eqref{e:Snontriv}. We can thus apply \thmref{thm:general}. By
  this theorem there are only two possibilities, either
  $\cS(t)=Mt$ or $\cS(t)=V_{\alpha}(t)$.
  Since in both cases $\cS(0)=0$,
  the convergence of $S_N$ actually holds in the $M_1$-topology.

  The possible limits
  $\mathcal S$ are in the subspace $D^1_{u,\uparrow\uparrow}$ of
  unbounded strictly increasing functions from $[0,\infty)$ to $\RR$, and their
  inverses are continuous. By \cite[Corollary~13.6.4]{W02}, the inverse
  map from the space $D^1_{u,\uparrow}(M_1)$ of unbounded non-decreasing
	functions to $D^1(J_1)$ is continuous at
  $D^1_{u,\uparrow\uparrow}$, therefore $S_N^{-1}$ converge to $\cS^{-1}$
  in $\PP$-distribution on $D^1(J_1)$. Moreover, the rescaled random
  walks $Y_N$ converge in $\PP$-distribution on $D^d(J_1)$ to a standard
  $d$-dimensional Brownian motion $B$.

  To proceed, we need
  to show that $B$ and the limit clock process $\cS$ are independent.
	This is trivial for the case $\cS(t)=Mt$, so we may assume that $\cS=V_{\alpha}$.
  We will use \cite[Lemma~15.6]{K02} which applied to our
  situation states that if
  $B,\cS$ are such that $B(0)=\cS(0)=0$ and the process
  $(B,\cS)\in D^{d+1}$ has independent increments and no fixed jumps, $\cS$
  is a.s.~a step process and $\Delta B \cdot \Delta \cS = 0$ a.s.,~then
  $B$ and $\cS$ are independent. The only assumption that remains to be
  verified is that $(B,\cS)$ has jointly independent increments.

  For fixed times $0=t_0<t_1<\cdots<t_n=t$, consider the version
  $\tilde{S}(\floor{Nt})$ from \eqref{def:stilde} in the proof
  of \lemref{lem:indincr}. We have seen that every increment
  $\tilde{S}(\floor{Nt_i})-\tilde{S}(\floor{Nt_{i-1}})$ is independent of
  the increments $\{\xi_k:~k\notin[\floor{Nt_{j-1}},\floor{Nt_j}-1]\}$
	of the random walk $Y$. Since there is
  such a version of $\tilde{S}(\floor{Nt})$ for every choice of times $t_j$,
  and every such $\tilde{S}(\floor{Nt})$ converges to $\cS$ after
  normalization, we obtain that for the limit $\cS$ every increment
  $\cS(t)-\cS(s)$ is independent of $\{B(u):~u\notin[t,s]\}$. Since both
  $B$ and $\cS$ have independent increments, this implies that $(B,\cS)$
  has jointly independent increments. Applying \cite[Lemma~15.6]{K02}
  it follows that the two limit processes $B$
  and $\cS$, and thus also $B$ and $\cS^{-1}$ are
  independent.

  It follows that $(Y_N,S_N^{-1})$ converge in distribution
  on $D^d(J_1)\times D^1_{u,\uparrow}(J_1)$ to $(B,\cS^{-1})$. By
  \cite[Theorem~13.2.2]{W02}, the composition map from $D^d(J_1)\times
  D^1_{u,\uparrow}(J_1)$ to $D^d(J_1)$ taking $\left(y(t),s(t)\right)$
  to $y\left(s(t)\right)$ is continuous at $(y,s)$ if $y$ is continuous
  and $s$ non-decreasing. From this we conclude that the compositions
  $X_N(t)=Y_N(S_N^{-1}(t))$ converge in distribution on $D^d(J_1)$
  to $B(\cS^{-1}(t))$ as required.
\end{proof}

For the proof of \propref{prop:convclock} we will relate the clock process
$S$ to the quadratic variation process of the RTRW $X$ and then apply
\cite[Corollary~VI.6.29]{JS03} which states that under some conditions,
whenever a sequence of processes converges in distribution, then so does
the sequence of their quadratic variations.

We need some definitions first. For a $d$-dimensional
pure-jump process $Z$, let $Z^{(i)}$ denote the $i$-th coordinate of $Z$,
and let $\Delta Z^{(i)}(t) = Z^{(i)} (t)- Z^{(i)}(t-)$ be the jump size
of $Z^{(i)}$ at time $t$. The quadratic variation process $[Z,Z]_t$ is
a $d\times d$ matrix-valued process, where the $(i,j)$-th entry is the
quadratic covariation of the $i$-th and $j$-th coordinate of $Z$, which is
\begin{equation*}
  [Z^{(i)},Z^{(j)}]_t = \sum_{0<s\leq t} \Delta
  Z^{(i)}(s)\Delta Z^{(j)}(s).
\end{equation*}

We proceed by relating the inverse $S_N^{-1}$ of the clock process
to the quadratic variation process of $X_N$.
\begin{lemma}
  \label{lem:quadvar}
  Under Assumption~\ref{ass:BM}, let $[X_N,X_N]_t$ be the quadratic
	variation process of $X_N$, and define
  $\sigma^2=\EE[|\cA \xi_j|^2]$ (recall \eqref{e:bmlimit} and \eqref{e:XN}
	for the notation). Then for every $t>0$,
  \begin{equation*}
    \frac{
      \operatorname{trace} [X_N,X_N]_t}{\sigma^2 S_N^{-1}(t) }
    \xrightarrow{N\to\infty} 1 \text{ in } \PP\text{-probability}.
  \end{equation*}
\end{lemma}

\begin{proof}
  Easy computation yields
  \begin{equation*}
    \operatorname{trace}[X_N,X_N]_t = \sum_{i=1}^d
    \sum_{0<s\leq t} (\Delta X_N^{(i)}(s))^2 =  \frac{1}{N}
    \sum_{j\leq S^{-1}(a_N t)} |\cA \xi_j|^2.
  \end{equation*}
  The process $S^{-1}$ has increments of size 1, and since the
  times between increments are a.s.~finite,
  $S^{-1}(a_N t) \nearrow \infty$ a.s.~as $N\to\infty$. Therefore, since
  $\sigma^2=\EE[|\cA \xi_j|^2]<\infty$ by Assumption~\ref{ass:BM}, the
  law of large numbers implies
  \begin{equation*}
    \PP\left[\left|\frac{N}{S^{-1}(a_N t)} \operatorname{trace}
		[X_N,X_N]_t - \sigma^2 \right| > \epsilon\right] \longrightarrow 0
		\text{ as } N\to\infty \text{ for every } \epsilon>0.
  \end{equation*}
  Noting that $\frac{1}{N}S^{-1}(a_N t) = S^{-1}_N(t)$ finishes the proof.
\end{proof}

We now check that the assumptions for \cite[Corollary VI.6.29]{JS03}
are fulfilled.

\begin{lemma}
  \label{lem:martingale}
  If Assumption~\ref{ass:BM} holds, then the rescaled processes $X_N$ are
	local martingales (with respect to the natural filtration $\sigma(X_N(t))$)
	with bounded increments.
\end{lemma}

\begin{proof}
  The increments are bounded by Assumption~\ref{ass:BM}. The local
  martingale property is unaffected by linear scaling, it is hence
  sufficient to prove it for the process $X$.

  We show that the sequence of
  stopping times $\sigma_l=S(l)$, $l\geq1$, is a localizing sequence
  for $X$, i.e.~we show that $(X(t\wedge \sigma_l))_{t\geq0}$ is a
  martingale for every $l\geq1$.

  We introduce the filtration $\cF_t=\sigma(Y(k), S(k):~k\leq t)$.
  Obviously, $Y$ is an $\cF$-martingale, and $S^{-1}(t)$ is an
	$\mathcal F$-stopping time for every $t\ge 0$, with $S^{-1}(t)\geq
	S^{-1}(s)$ for $t\geq s$. The natural filtration for $X$, $\cG_t =
	\sigma(X(t))$ satisfies $\cG_t = \cF_{S^{-1}(t)}$ and is right-continuous
	(see \cite[Proposition~7.9]{K02}).
  The sequence of random variables $\sigma_l$ is indeed an increasing sequence
  of $\cG$-stopping times ($\sigma_l$ is the time at which the process $X$
	jumps for the $l$-th time). Moreover, by definition
  $S^{-1}(t\wedge \sigma_l) = S^{-1}(t) \wedge S^{-1}(\sigma_l) \leq
	S^{-1}(\sigma_l)=l$.
  Applying Doob's optional sampling theorem
  (see e.g.~\cite[Theorem~7.12]{K02}) to the discrete-time martingale $Y$
  and the bounded stopping time $S^{-1}(t\wedge \sigma_l)$, we obtain
  \begin{equation*}
    \EE\left[X(t\wedge \sigma_l) \mid \cG_s\right] =
    \EE\left[Y\left(S^{-1}(t\wedge \sigma_l)\right)\mid \cF_{S^{-1}(s)}\right]
    = Y\left(S^{-1}(t\wedge \sigma_l) \wedge S^{-1}(s)\right) = X(s\wedge
      \sigma_l).
  \end{equation*}
  This completes the proof.
\end{proof}

We can now prove \propref{prop:convclock}.

\begin{proof}[Proof of \propref{prop:convclock}]
  By \lemref{lem:martingale}, $X_N$ are local martingales with
  bounded increments.
  \cite[Corollary~VI.6.29]{JS03} then implies  that the quadratic variation
  processes $[X_N,X_N]_t$ converge component-wise on $D^1(J_1)$ to the
  quadratic variation process $[\cX,\cX]_t$ of $\cX$.
  Since all jumps of the processes $[X_N^{(i)},X_N^{(i)}]_t$,
  $i=1,\dots,d$, are positive, \cite[Theorem 12.7.3
    (continuity of addition at limits with jumps of common sign)]{W02}
  yields that  $\operatorname{trace}[X_N,X_N]_t$ converges to some
  non-decreasing process in $D^1(M_1)$. From \lemref{lem:quadvar} it then
  follows that the  inverses $S_N^{-1}$ of the rescaled clock processes
  converge to some non-decreasing process $\cS^{-1}(t)$ in $D^1(M_1)$.

  For non-decreasing functions $x\in D^1$ the right-continuous inverse
  satisfies $(x^{-1})^{-1}=x$, and thus $S_N= (S_N^{-1})^{-1}$. Hence, by
  \cite[Theorem~13.6.1]{W02}, which ensures the continuity of the inverse
  operation, $S_N$ converges to $\cS$ in $D^1(M_1)$ provided that
  ${\cS(0)=(\cS^{-1})^{-1}(0)=0}$.

  If we do not know whether $\cS(0)=0$, this theorem does
  not apply. This issue can be solved by weakening the topology
  from $M_1$ to $M_1'$ (see \cite[Section~13.6.2]{W02} for details). In
  particular, \cite[Theorem~13.6.2]{W02} yields that $S_N$
  converge to $\cS$ in distribution in $D^1(M_1')$.
\end{proof}


\section{Proofs of sufficiency criteria}
\label{sec:proofsuff}

\thmref{thm:brown}, giving a sufficient criterion for convergence to
Brownian motion, is an immediate consequence of the ergodicity of the
sequence of successive waiting times.

\begin{proof}[Proof of \thmref{thm:brown}]
  Consider $\tilde{\tau}=(\tilde{\tau}_k)_{k\geq0}$ and let $\theta:~\RR^{\NN}\to\RR^{\NN}$ be the
  left-shift along the sequence, which by \lemref{lem:ergod} acts ergodically along $\tilde{\tau}$.

  If $\EE[\tilde{\tau}_0]=M$ is finite, the function
  $f(\tilde{\tau})=\tilde\tau_0$ is integrable, and we can apply the
  ergodic theorem to $f$ to get
  \begin{align*}
    \lim_{N\to\infty} \frac{1}{N} S(\floor{Nt}) &=
    \lim_{N\to\infty} \frac{1}{N} \sum_{k=0}^{\floor{Nt}-1} \tilde{\tau}_k
    = \lim_{N\to\infty} t \frac{1}{Nt} \sum_{k=0}^{\floor{Nt}-1}
    f(\theta^k(\tilde{\tau})) \\
    &= t \EE\left[f(\tilde{\tau})\right] = Mt \text{ almost
      surely}.
  \end{align*}
  Thus we have that the rescaled clock processes $S_N$ converge in
  distribution on $D^1(J_1)$ to $Mt$, where the normalization is $a_N=N$.
  If additionally Assumption~\ref{ass:BM} holds, using the same arguments
  as in the proof of \thmref{thm:class} we conclude that the $X_N$
  converge and the limit $\cX$ is as in case $(i)$ of \thmref{thm:class}.
\end{proof}


Before starting the proof of \thmref{thm:fk}, which deals with the
convergence to the Fractional Kinetics, we briefly
sketch some examples that illustrate how different functions $f$ in
condition \eqref{cond:fk} arise.

First, consider the CTRW defined in \eqref{e:ctrw}. The waiting times
$\tau_x^i$ of this model lie in the domain of attraction of an $\alpha$-stable
law, that is there is a slowly varying function $\ell_0$ (in general
  different from $\ell$ of \eqref{e:ctrw}) such that the sum of $N$
independent waiting times normalized by $a_N=N^{1/\alpha}\ell _0(N)$
converges to an $\alpha$-stable random variable, see
e.g.~\cite[Theorem~4.5.1]{W02}. Thus the quenched Laplace transform
(which is deterministic here) satisfies
\begin{equation*}
  \hat{\pi}_0(\lambda/a_N) = \exp\left\{-c'\lambda^{\alpha}{N}^{-1}(1+o(1))
		\right\}
  \quad \text{ as } N\to\infty
\end{equation*}
for some $c'>0$. Taking this to the power $r\ell^{*}(N)$ it follows that the
CTRW satisfies condition \eqref{cond:fk} with $a_N=N^{1/\alpha}\ell_0(N)$
and $f(r)=r$.

Secondly, consider the following simplified Bouchaud trap model
(cf.~\eqref{e:btm1}). Let $\pi_x=\delta_{\tau_x}$ where the $\tau_x$,
$x\in \mathbb Z^d$,
are heavy-tailed i.i.d.~random variables, that is
\begin{equation*}
	\PP[\tau_x > u] = cu^{-\alpha}(1+o(1)) \text{ as } u\to\infty.
\end{equation*}
Then the quenched Laplace transform satisfies
\begin{equation*}
	\hat{\pi}_0(\lambda/a_N) = \exp\{-\lambda a_N^{-1} \tau_0\}.
\end{equation*}
Taking this to the power $r\ell^{*}(N)$ and taking the expectation over
$\tau_0$, this is the Laplace transform of a random variable in the
normal domain of attraction of an $\alpha$-stable law, evaluated at
$r \lambda  \ell^{* }(N)/a_N$. By normal domain of attraction we
mean that the sum of $N$ independent such random variables normalized by
$c'N^{1/\alpha}$ converges to an $\alpha$-stable random variable, see
e.g.~\cite[Theorem~4.5.2]{W02}. Thus choosing
$a_N = c'N^{1/\alpha} \ell^{*}(N)^{1-1/\alpha}$, the Laplace transform is
\begin{align*}
  \EE\left[\hat{\pi}_0(\lambda/a_N)^{r\ell^{*}(N)}\right]
	&= \EE\left[\exp\left\{-\frac{\lambda r}{c'\ell^{*}(N)^{-1/\alpha}}
		\frac{\tau_x}{N^{1/\alpha}}\right\}\right]\\
  &= \exp\left\{-c''\frac{\ell^{*}(N)}{N}\lambda^{\alpha} r^{\alpha}
		(1+o(1))\right\}
  \text{ as } N\to\infty.
\end{align*}
Condition \eqref{cond:fk} is thus satisfied for $f(r)=r^{\alpha}$.

To see that $f(r)$ can be more than just a power of $r$, consider the
following mixture of the above two models.  To this end, let us fix the
slowly varying function $\ell$ of \eqref{e:ctrw} so that the
normalization $a_N=N^{1/\alpha }\ell_0(N)$ of the first example agrees
with the normalization  $a_N=c'N^{1/\alpha }\ell^*(N)^{1-1/\alpha }$ of
the second example. (This is possible e.g.~when $
  1/\ell^*(N)\to\gamma\in(0,1)$, then also $\ell$ converges to a
positive constant, or when $\ell^*(N)\sim c\log^{-1}N$, as is the case for simple
random walk on $\ZZ^2$, then $\ell(N) \sim c' \ell^*(N)^{\alpha-1}$.)
The mixture is now defined as follows. For some $p\in(0,1)$, let each
$\pi_x$ with probability $p$ be a heavy-tailed distribution as in
\eqref{e:ctrw}, and with probability $1-p$, let $\pi_x$ be
$\delta_{\tau_x}$ where the $\tau_x$ are heavy-tailed random variables with
\begin{equation*}
  \PP[\tau_x > u] = cu^{-\alpha}(1+o(1)) \text{ as }u\to\infty.
\end{equation*}
Then, by combining the arguments above, condition
\eqref{cond:fk} is satisfied with the normalization
$a_N=c'N^{1/\alpha} \ell^{*}(N)^{1-1/\alpha}$ and
$f(r)=pr + (1-p) r^{\alpha}$.

\begin{proof}[Proof of \thmref{thm:fk}]
  By Theorem~\ref{thm:general} it is sufficient to show that
  \begin{equation}
    \label{e:critgoal}
    \lim_{N\to\infty}\mathbb E[\exp\{-\lambda S_N(t)\}] = e^{-c t\lambda^\alpha }
  \end{equation}
  for some $c\in (0,\infty)$, this is equivalent to convergence of $S_N$ to an
	$\alpha$-stable subordinator. Using the independence of the $\pi_x$'s,
  recalling that $\hat \pi_x $ denotes the Laplace transform of $\pi_x$,
  we have
  \begin{equation}
    \begin{split}
      \label{e:crita}
      \mathbb E\Big[\exp\Big\{-\frac{\lambda}{a_N} S(\floor{Nt})\Big\}\Big|Y\Big]
      &=\mathbb E\Big[\exp\Big\{-\frac{\lambda}{a_N}
          \sum_{x\in \mathbb Z^d}\sum_{i=1}^{L(x,\floor{Nt}-1)}\tau_x^i
          \Big\}\Big|Y\Big]
      \\&=\prod_{x\in \mathbb Z^d}\mathbb E\Big[\hat \pi_x({\lambda}/{a_N})
          ^{L(x,\floor{Nt}-1)}\Big\}\Big|Y\Big].
    \end{split}
  \end{equation}

  Treating the case when $Y$ is transient first, let
  $R^k(Nt)=\{x\in \mathbb Z^d: L(x,\floor{Nt}-1)=k\}$. By
  Lemma~\ref{lem:llnrangegeneral},
  $|R^k(Nt)|/(Nt)\xrightarrow{N\to\infty} \gamma^2 (1-\gamma )^{k-1}$ in
  probability. Using the translation invariance, the right-hand side of
  \eqref{e:crita} can be written as
  \begin{equation*}
    \exp\Big\{\sum_{k=1}^\infty |R^k(Nt)|
      \log \mathbb E\big[\hat \pi_0(\lambda /a_N)^k\big]\Big\}.
  \end{equation*}
  For arbitrary $M\in \mathbb N$, using the law of large numbers for
  $|R^k(Nt)|$ and assumption \eqref{cond:fk} with the continuity of $f$,
  \begin{equation}
    \label{e:critd}
    \sum_{k=1}^M |R^k(Nt)|
    \log \mathbb E\big[\hat \pi_0(\lambda /a_N)^k\big]
    \xrightarrow{N\to\infty}
    - t \lambda^\alpha \sum_{k=1}^M
    f(k \gamma)\gamma (1-\gamma )^{k-1},
  \end{equation}
  in probability. Applying Jensen's inequality, it is easy to see that
  $f(k)$ grows at most linearly with $k$, so the right-hand side of the
  above expression converges as $M\to\infty$ to a finite value, by
  assumptions of the theorem. On the other hand, by Jensen's inequality
  again,
  for every $\delta >0$
  \begin{equation}
    \begin{split}
      \label{e:critb}
      \mathbb P&\Big[-\sum_{k=M}^\infty|R^k(Nt)|
        \log \mathbb E\big[\hat \pi_0(\lambda /a_N)^k\big]\ge \delta \Big]
    \\& \le
      \mathbb P\Big[-\log  \mathbb E\big[\hat \pi_0(\lambda /a_N)\big]
				\sum_{k=M}^\infty|R^k(Nt)| k\ge \delta \Big].
    \end{split}
  \end{equation}
  By the Markov inequality, for $0<c_1< -\log (1-\gamma )$,
    $\mathbb P[|R^k(Nt)|/(Nt)\ge e^{-c_1 k}] \le e^{-c'k}$
  uniformly for all $k\ge M$ and $N$ large enough, and thus by a union bound
  \begin{equation}
    \label{e:critc}
    \PP\big[\exists k\ge M \text{ such that } |R^k(Nt)|/ (Nt)\ge e^{-c_1k }\big]
		\le C e^{-c'M}
  \end{equation}
  uniformly in $N$. Using \eqref{e:critc} and the fact that
  $\log  \mathbb E\big[\hat \pi_0(\lambda /a_N)\big]$ is finite by
  assumption, it follows that the left-hand side of \eqref{e:critb}
  converges to $0$ in probability when $N\to\infty$ and then $M\to
  \infty$, and therefore \eqref{e:critd} also holds with $M=\infty$.
  Using the bounded convergence theorem, it then follows that
  \begin{equation*}
    \begin{split}
      \mathbb E&\Big[\exp\Big\{-\frac{\lambda}{a_N} S(\floor{Nt})\Big\}\Big] =
      \mathbb E\Big[\exp\Big\{\sum_{k=1}^\infty |R^k(Nt)|
        \log \mathbb E\big[\hat \pi_0(\lambda /a_N)^k\big]\Big\}\Big]
    \\&\xrightarrow{N\to\infty} \exp\Big\{- t \lambda^\alpha \sum_{k=1}^\infty
      f(k\gamma)\gamma(1-\gamma )^{k-1}\Big\},
    \end{split}
  \end{equation*}
  which proves \eqref{e:critgoal} in the transient case.

  To treat the recurrent case, we fix $\beta >0$ small and define for
  $k\ge1$
  \begin{equation*}
    R_\beta^k(Nt)=\{x\in \mathbb Z^d: (k-1)\beta\ell^*(N) <L(x,\floor{Nt}-1)\le
      k\beta\ell^*(N) \}.
  \end{equation*}
  By Lemma~\ref{lem:llnrangegeneral},
  $|R_\beta^k(Nt)|\ell^*(N)/(Nt)\xrightarrow{N\to\infty} e^{-(k-1)\beta} -e^{-k\beta }$ in
  probability. The right-hand side of \eqref{e:crita} can be
  bounded from above by
  \begin{equation*}
    \exp\Big\{\sum_{k=1}^\infty |R_\beta^k(Nt)|
      \log \mathbb E\big[\hat \pi_0(\lambda /a_N)^{\beta(k-1)\ell^*(N)}\big]\Big\},
  \end{equation*}
  and from below by
  \begin{equation*}
    \exp\Big\{\sum_{k=1}^\infty |R_\beta^k(Nt)|
      \log \mathbb E\big[\hat \pi_0(\lambda /a_N)^{\beta k\ell^*(N)}\big]\Big\}.
  \end{equation*}
  Following the same steps as in the transient case, it can be easily shown that
  \begin{equation*}
    \begin{split}
      &\exp\Big\{-t \lambda^\alpha  \sum_{k=1}^\infty f(\beta
          (k-1))\big(e^{-(k-1)\beta}
          -e^{-k\beta }\big)\Big\} \le
      \lim_{N\to\infty} \mathbb E\big[e^{-\lambda S_N(t)}\big]
      \\&\le
      \exp\Big\{-t \lambda^\alpha  \sum_{k=1}^\infty f(\beta k)\big(e^{-(k-1)\beta}
          -e^{-k\beta }\big)\Big\}
    \end{split}
  \end{equation*}
  Since $f$ is a monotone function, the sums in the above expression can
  be viewed as lower and upper Riemann sums for the integral
  $\int_{0}^\infty f(x)e^{-x}\,dx$ to which they tend when $\beta \to 0$.
  This integral is finite since as argued before $f$ grows at most linearly,
	and \eqref{e:critgoal} is proved in the recurrent case.
\end{proof}


\section{Ignoring small sets} 
\label{sec:proofs}

In this section we prove Lemma~\ref{lem:badtimes} which allows us to
ignore small sets when dealing with the clock process.

We first
assume that the random walk $Y$ is transient,  that is
$1/\ell^*(n)\to \gamma\in(0,1)$ as $n\to\infty$. We start by noting that
for every $x\in R(\floor{Nt})$ and $i\in\{1,\dots,L(x,\floor{Nt}-1)\}$,
since $(\tau_x^i)_{i\ge 1}$ are i.i.d.,
\begin{equation}
  \label{e:xtoone}
  \EE\left[\frac{\tau_x^i}
    {S(\floor{Nt})} ~\Big|~ Y \right] =
  \EE\left[\frac{\tau_x^1}{S(\floor{Nt})}
    ~\Big|~ Y \right].
\end{equation}
For fixed $0\leq l< \floor{Nt}$,
let $x=Y(l)$ and $i=L(Y(l),l)$, that is $\tilde \tau_l = \tau_x^i$. Using
\eqref{e:xtoone} and the fact that $(\tau_x^1)_{x\in \mathbb Z^d}$ are
i.i.d.~under $\mathbb P$,
\begin{equation}
  \begin{split}
    \EE\left[\frac{\tilde{\tau}_l}{S(\floor{Nt})} ~\Big|~ Y \right]
    &=  \EE\left[\frac{\tau_x^i}{\sum_{y\in R(\floor{Nt})}
        \sum_{j=1}^{L(x,\floor{Nt}-1)} \tau_y^j} ~\Big|~ Y \right] \\ &\leq
    \EE\left[\frac{\tau_x^1}{\sum_{y\in R(\floor{Nt})} \tau_y^1} ~\Big|~
      Y \right] \\ &= \frac{1}{|R(\floor{Nt})|} \sum_{z\in R(\floor{Nt})}
    \EE\left[\frac{\tau_z^1}{\sum_{y\in R(\floor{Nt})} \tau_y^1} ~\Big|~ Y
    \right] \\ &= \frac{1}{|R(\floor{Nt})|}.
  \end{split}
  \label{eq:boundfirst}
\end{equation}
By the law of large numbers for $R(n)$ (\lemref{lem:llnrangegeneral})
in the transient case, there is a constant $C<\infty$ such that
for all $N$ large enough
\begin{equation*}
  \PP\left[ |R(\floor{Nt})| < CN\right ] <
  \epsilon.
\end{equation*}
Hence, for $N$ large enough,
\begin{align*}
  \PP\bigg[\sum_{l\in \cB}& \tilde{\tau}_l \geq \delta
    S(\floor{Nt}),~|\cB|\leq \epsilon N \bigg] \\
  &\leq \PP\bigg[\sum_{l\in \cB} \tilde{\tau}_l \geq \delta
    S(\floor{Nt}) ,~|\cB|\leq \epsilon N,~|R(\floor{Nt})|\geq
    CN\bigg] + \epsilon.
  \intertext{Using the Markov inequality and \eqref{eq:boundfirst}, this is bounded from above by}
  &\le  \frac{1}{\delta} \EE\left[ \sum_{l\in \cB}
    \EE\left[\frac{\tilde{\tau}_l}{S(\floor{Nt})} ~\Big|~ Y \right]
    \Ind{|\cB|\leq \epsilon N} \Ind{|R(\floor{Nt})| \geq CN} \right]
  +\epsilon. \\
  &\leq  \frac{1}{\delta} \EE\left[\frac{|\cB| }{|R(\floor{Nt})|}
    \Ind{|\cB|\leq \epsilon N} \Ind{|R(\floor{Nt})|\geq CN} \right]
  +\epsilon
  \\&\leq \frac{\epsilon}{C\delta} +\epsilon.
\end{align*}
Letting $N\to\infty$ and then $\epsilon\to0$ completes the proof of the
lemma in the transient case.

We now consider the recurrent case. Let
$R_{\cB}= \{Y(l):l\in \mathcal B\}$, and
for $x\in R_{\cB}$ let $L_{\cB}(x)=|\{l\in \mathcal B: Y(l)=x\}|$.
Fix some small $\beta>0$ and let
\begin{equation*}
  \begin{split}
    R_{>\beta}&=\{x\in R(\floor{Nt}):L(x,\floor{Nt}-1) > \beta \ell^*(N) \},
    \\R_{\leq\beta}&=\{x\in R(\floor{Nt}):L(x,\floor{Nt}-1) \leq \beta \ell^*(N)
      \}.
  \end{split}
\end{equation*}
By \lemref{lem:llnrangegeneral}, the sizes of $R_{>\beta }$ and
$R_{\le \beta }$ satisfy weak laws of large numbers with respective
averages $Nte^{-\beta }/\ell^*(N)(1+o(1))$ and
$Nt(1-e^{-\beta })/\ell^*(N)(1+o(1))$.
In particular for $C_{\beta} = (1-\epsilon) e^{-\beta}t$ and $c_{\beta} =
(1+\epsilon)(1-e^{-\beta})t$, for all $N$ large enough,
\begin{equation*}
  \PP\left[|R_{>\beta}| < C_{\beta} \frac{N}{\ell^*(N)}
  \right] + \PP\left[|R_{\leq\beta}| > c_\beta \frac{N}{\ell^*(N)} \right]
  \leq \epsilon.
\end{equation*}
Therefore, for $N$ large enough,
\begin{align}
  &\PP\bigg[\sum_{l\in \cB} \tilde{\tau}_l \geq \delta
    S(\floor{Nt}),~|\cB|\leq \epsilon N \bigg] \nonumber\\
  &\leq \PP\bigg[\sum_{x\in R_{\cB}\cap R_{>\beta}} \sum_{i=1
    }^{L_{\cB}(x)} \tau_x^i \geq \frac{\delta}{2}
    S(\floor{Nt}),~|\cB|\leq \epsilon N ,~|R_{>\beta}|\geq
    C_{\beta}\frac{N}{\log N}\bigg] \label{eq:awfulthing1} \\
  &+ \PP\bigg[\sum_{x\in R_{\cB}\cap R_{\leq\beta}}
    \sum_{i=1}^{L_{\cB}(x)} \tau_x^i \geq
    \frac{\delta}{2} S(\floor{Nt}),~|R_{>\beta}|\geq
    C_{\beta}\frac{N}{\log N},~|R_{\leq\beta}| \leq  c_\beta
    \frac{N}{\log N}\bigg] +\epsilon.\label{eq:awfulthing2}
\end{align}
Using \eqref{e:xtoone} and
the similar reasoning as in the transient case, since
$(\sum_{i=1}^{\beta\ell^*(N)} \tau_x^i)_{x\in \mathbb Z^d}$
are i.i.d.~with respect to the annealed measure
and independent of $Y$, we have for
$x\in R_{\cB}\cap R_{>\beta}$,
\begin{equation*}
  \begin{split}
    \EE\left[ \frac{
        \sum_{i=1}^{L_{\cB}(x)} \tau_x^i}{ S(\floor{Nt})}
      ~\Big|~ Y \right]
    &=
    \frac{L_{\cB}(x)}{\beta\ell^*(N)} \EE\left[ \frac{
        \sum_{i=1}^{\beta\ell^*(N)} \tau_x^i}{ S(\floor{Nt})} ~\Big|~ Y \right]
    \\&\leq \frac{L_{\cB}(x)}{\beta\ell^*(N)} \EE\left[
      \frac{ \sum_{i=1}^{\beta\ell^*(N)} \tau_x^i}{ \sum_{y \in
          R_{>\beta}} \sum_{i=1}^{\beta \ell^*(N)} \tau_y^i } ~\Big|~
      Y \right]
    \\&= \frac{L_{\cB}(x)}{|R_{>\beta}|\beta\ell^*(N)}.
  \end{split}
\end{equation*}
Therefore, using the Markov inequality,
\begin{equation}
  \begin{split}
    \label{e:bound1}
    \eqref{eq:awfulthing1}&\le
    \frac{2}{\delta} \EE\left[ \sum_{x\in R_{\cB}\cap R_{>\beta}}
      \EE\left[ \frac{ \sum_{i=1}^{L_{\cB}(x)} \tau_x^i}{
          S(\floor{Nt})} ~\Big|~ Y \right] \Ind{|\cB|\leq\epsilon N}
      \Ind{|R_{>\beta}|\geq C_{\beta}\frac{N}{\ell^*(N)}} \right] \\
    &\leq \frac{2}{\delta} \EE\left[
      \sum_{x\in R_{\cB}\cap R_{>\beta}}
      \frac{L_{\cB}(x)}{|R_{>\beta}|\beta\ell^*(N)}
      \Ind{|\cB|\leq\epsilon N} \Ind{|R_{>\beta}|\geq
        C_{\beta}\frac{N}{\ell^*(N)}} \right] \\
    &\leq \frac{2\epsilon}{\delta\beta C_{\beta}}.
  \end{split}
\end{equation}
where for the last inequality we used the fact that
$\sum_{x} L_{\cB}(x)\le |\cB| \le \varepsilon N$.

It remains to bound \eqref{eq:awfulthing2}. Using again the fact that
$(\sum_{i=1}^{\beta\ell^*(N)} \tau_x^i)_{x\in \mathbb Z^d}$
are i.i.d.~with respect to the annealed measure and independent of $Y$,
\begin{align*}
  \PP&\bigg[\sum_{x\in R_{\cB}\cap R_{\leq\beta}}
    \sum_{i=1}^{L_{\cB}(x)} \tau_x^i \geq \frac{\delta}{2}
    S(\floor{Nt}) ~\Big|~ Y \bigg] \\
  &\leq \PP\bigg[\left(1+\frac{\delta}{2}\right) \sum_{x\in
      R_{\cB}\cap R_{\leq\beta}} \sum_{i=1}^{\beta\ell^*(N)} \tau_x^i
    \geq \frac{\delta}{2} \sum_{x\in R_{>\beta}\cup(R_{\cB}\cap
        R_{\leq\beta})} \sum_{i=1}^{\beta\ell^*(N)} \tau_x^i ~\Big|~
    Y \bigg] \\
  &\leq \frac{2+\delta}{\delta} \EE\bigg[ \frac{ \sum_{x\in
        R_{\cB}\cap R_{\leq\beta}} \sum_{i=1}^{\beta\ell^*(N)}
      \tau_x^i} {\sum_{x\in R_{>\beta}\cup(R_{\cB}\cap R_{\leq\beta})}
      \sum_{i=1}^{\beta\ell^*(N)} \tau_x^i} ~\Big|~ Y \bigg] \\
  &= \frac{2+\delta}{\delta} \frac{|R_{\cB}\cap
    R_{\leq\beta}|}{|R_{>\beta}\cup(R_{\cB}\cap
      R_{\leq\beta})|}.
\end{align*}
Therefore,
\begin{equation}
  \label{e:bound2}
  \eqref{eq:awfulthing2}
  \leq \frac{2+\delta}{\delta} \frac{c_{\beta}}{C_{\beta}}
  = \frac{2+\delta}{\delta} \frac{1+\epsilon}{1-\epsilon}
  \left(e^{\beta}-1\right).
\end{equation}

Combining \eqref{eq:awfulthing1}--\eqref{e:bound2} and letting
$N\to\infty$, then $\epsilon\to0$ and finally $\beta \to 0$ finishes the proof of the lemma in
the recurrent case.\qed


\appendix
\section{Laws of large numbers for range-like objects} \label{appendix}
We prove here that Assumption~\ref{ass:slowvar} implies weak laws of large
numbers for several range-related quantities.
The proofs are based on the classical paper \cite{DE51},
see also \cite[Chapter 21]{Rev}.

Recall that
\begin{equation*}
  R(n)=\{x\in \ZZ^d:~L(x,n-1)>0\}
\end{equation*}
is the range of the random walk $Y$ up to time $n-1$. In the recurrent
case, i.e.~if $\ell^{*}(n)\to\infty$, define for $k\geq1$ and $\beta>0$
\begin{equation*}
  R^k_{\beta}(n)=\{x\in \ZZ^d:~ L(x,n-1) \in ((k-1),k]
    \beta\ell^*(n)\}
\end{equation*}
the set of vertices visited $(k-1)\beta\ell^*(n)$ to $k\beta\ell^{*}$ times
up to time $n-1$. In the transient case, if $1/\ell^{*}(n)\to\gamma \in (0,1)$, let for $k\geq1$
\begin{equation*}
  R^k(n) = \{x\in\ZZ^d:~L(x,n-1) = k\}
\end{equation*}
the vertices visited exactly $k$ times up to time $n-1$.

We say that a
sequence of random variables $Z_n$ satisfies the weak law of large numbers if
$Z_n/EZ_n\xrightarrow{n\to\infty}1$ in probability.

\begin{lemma}
  \label{lem:llnrangegeneral}
  \begin{enumerate}[(i)]
    \item
    If Assumption~\ref{ass:slowvar} holds, then
    $|R(n)|$ satisfies the weak law of large numbers with
    \begin{equation*}
      \mathbb E[|R(n)|] = \frac{n}{\ell^*(n)}(1+o(1)) \text{
        as }n\to\infty.
    \end{equation*}

    \item
    If in addition $\ell^*(n)\to \infty$ as $n\to\infty$, then
    $|R^k_\beta (n)|$ satisfies the weak law of large numbers for every
    $k\ge 1$ and $\beta >0$, and
    \begin{equation*}
      \mathbb E[|R^k_{\beta}(n)|] =
      (e^{-(k-1)\beta}-e^{-k\beta})\frac{n}{\ell^*(n)}(1+o(1))
      \text{ as }n\to\infty.
    \end{equation*}

    \item
    If, on the other hand, $1/\ell^{*}(n)\to \gamma \in(0,1)$,
    then $|R^k(n)|$ satisfies the weak law of large numbers for every
    $k\ge 1$, and
    \begin{equation*}
      \mathbb E[|R^k(n)|] = \gamma^2(1-\gamma)^{k-1} n (1+o(1))
			\text{ as }n\to\infty.
    \end{equation*}
  \end{enumerate}
\end{lemma}

\begin{proof}
	Note that for the simple random walk in
	$d\geq3$ and $d=2$ respectively, part (i) is a classical result
	from \cite{DE51}, part (iii) was hinted at in \cite[Theorem
	12]{ET60} and proved in \cite{P74}, whereas part (ii) is
	a direct consequence of \cite[Theorem 4]{DE51} and \cite[Theorem 2]{C07}.
  Part (i) above is proved exactly as in \cite{DE51}. We include its proof, since
  proofs of (ii) and (iii) are its extensions.
  Let $\psi_k$ be the indicator of the event that a new vertex is found at
  time $k$,
  \begin{equation*}
    \psi_k=\Ind{Y(l)\neq Y(k)
      \text{ for all } 0\leq l<k},
  \end{equation*}
  with $\psi_0=1$. Recall that $\xi_i$
  denote the i.i.d.~increments of the random walk $Y$. Then,
  \begin{equation}
    \label{eq:psirk}
    \begin{split}
      \mathbb E[\psi_k] &= \mathbb P\left[Y(k)\neq
        Y(k-1),~Y(k) \neq Y(k-2),\dots,~Y(k)\neq Y(0)\right] \\
      &= \mathbb P\left[\xi_k\neq 0,~\xi_k+\xi_{k-1}\neq
        0,\dots,~\xi_k+\cdots+\xi_1\neq 0\right] \\
      &= \mathbb P\left[\xi_1\neq 0,~\xi_1+\xi_{2}\neq
        0,\dots,~\xi_1+\cdots+\xi_{k}\neq 0\right] \\
      &= \mathbb P[Y(l)\neq 0 \text{ for } l=1,\dots,k] = r_k.
    \end{split}
  \end{equation}
  For a slowly varying function $\ell$,
  $\sum_{k=1}^n \ell(k) = n \ell(n)(1+o(1))$ as $n\to\infty$ (see
    e.g.~\cite[p.~55]{Sen76}). Therefore, by Assumption~\ref{ass:slowvar},
  \begin{equation}
    \label{e:ER}
    \mathbb E[|R(n)|] =
    \sum_{k=0}^{n-1} \mathbb E[\psi_k] =
    \frac{n}{\ell^*(n)}(1+o(1)) \text{ as }n\to\infty.
  \end{equation}

  To prove the weak law of large numbers,  we compute the variance.
  First note that for $i\leq j$, by the Markov property,
  \begin{equation}
    \begin{split}
      \mathbb E[\psi_i \psi_j] &=
      \mathbb E\left[\Ind{Y(l)\neq Y(i),~0\leq l<i} \Ind{Y(l)\neq Y(j),~0\leq
          l<j}\right] \\
      &\leq \mathbb E\left[\Ind{Y(l)\neq Y(i),~0\leq l<i} \Ind{Y(l)\neq
          Y(j),~i\leq l<j}\right] = \mathbb E[\psi_i] \mathbb E[\psi_{j-i}].
    \end{split}
    \label{eq:exppsipsi}
  \end{equation}
  Then,
  \begin{equation}
    \begin{split}
      \Var{|R(n)|} &= \sum_{0\leq i,j\leq n-1} \mathbb E[\psi_i \psi_j]
      - \mathbb E[\psi_i] \mathbb E[\psi_j]  \\
      &\leq 2 \sum_{i=0}^{n-1} \sum_{j=i}^{n-1} \mathbb E[\psi_i]
      \left(\mathbb E[\psi_{j-i}] -\mathbb E[\psi_j] \right) \\
      &\leq 2 \sum_{i=0}^{n-1} \mathbb E[\psi_i] \left( \max_{k=0,\dots,n-1}
        \sum_{j=k}^{n-1}\mathbb E[\psi_{j-k}] -\mathbb E[\psi_j]
      \right). \label{eq:psipsi1}
    \end{split}
  \end{equation}
  By \eqref{eq:psirk},  $\mathbb E[\psi_k]$ is non-increasing,
  therefore the maximum in \eqref{eq:psipsi1} is attained in
  $k=\frac{n}{2}$. The parenthesis in \eqref{eq:psipsi1} can then be
  estimated using elementary properties of slowly varying functions,
  \begin{align*}
    \sum_{j=\frac{n}{2}}^{n-1}\mathbb E[\psi_{j-\frac{n}{2}}]
    -\mathbb E[\psi_j] &= \sum_{j=0}^{\frac{n}{2}-1} \frac{1}{\ell^*(j)} -
    \sum_{j=\frac{n}{2}}^{n-1} \frac{1}{\ell^*(j)} \\
    &= \sum_{j=0}^{\frac{n}{2}-1} \frac{1}{\ell^*(j)} - \left(
      \sum_{j=0}^{n-1} \frac{1}{\ell^*(j)} - \sum_{j=0}^{\frac{n}{2}-1}
      \frac{1}{\ell^*(j)}\right) \\
    &= 2 \frac{\frac{n}{2}}{\ell^*(\frac{n}{2})}(1+o(1)) -
    \frac{n}{\ell^*(n)} (1+o(1)) = \frac{n}{\ell^*(n)} o(1) \text{
      as } n\to\infty.
  \end{align*}
  Inserting this into \eqref{eq:psipsi1}, we obtain
  \begin{equation*}
    \Var{|R(n)|} \leq 2
    \sum_{i=0}^{n-1} \mathbb E[\psi_i] \frac{n}{\ell^*(n)} o(1) =
    o\left(\left(\frac{n}{\ell^*(n)}\right)^2\right)\text{ as }
    n\to\infty,
  \end{equation*}
  and the weak law of large numbers for $|R(n)|$ follows by usual
  arguments.

  Before turning to part (ii), we note the following fact on return
  times. Let as before $H_0^1=\inf\{i>0:~Y(i)=0\}$ denote the time of the first return
  to $0$, and $H_0^k = \inf\{i > H_0^{k-1}:~Y(i)=0\}$
  the time of the $k$-th return to $0$. Let $T_i=H_0^i-H_0^{i-1}$
  (with $H_0^0=0$) be the successive return times. By the Markov
  property the $(T_i)_{i\geq1}$ are i.i.d., and $\mathbb P[T_i > n] = r_n =
  \frac{1}{\ell^*(n)}$ by Assumption~\ref{ass:slowvar}. If $\ell^*(k) \to
  \infty$, the $T_i$ are a.s.~finite and have slowly varying tail. It is
  well known (e.g.~\cite[Theorem~3.2]{Dar52}) that for such i.i.d.~random
  variables $T_i$,
  \begin{equation}
    \frac{\sum_{i=1}^n T_i}{\max_{i=1}^n T_i} \to 1 \text{
      in probability as } n\to \infty. \label{eq:summaxt1}
  \end{equation}
  Since $\ell^*(cn) \sim \ell^*(n)$ as $n\to\infty$,
  \begin{equation}
    \mathbb P\big[ \max \{T_i:1\le i\le \beta\ell^*(n)\}\leq cn\big]
    = \left(1-\frac{1}{\ell^*(cn)}\right)^{\beta\ell^*(n)} =
    e^{-\beta}(1+o(1))
    \text{ as } n\to\infty. \label{eq:summaxt2}
  \end{equation}
  From \eqref{eq:summaxt1} and \eqref{eq:summaxt2} we obtain for every $c>0$
  and $\beta>0$
  \begin{equation}
    \mathbb P\left[ L(0,cn) \geq \beta\ell^*(n) \right]
    = \mathbb P\left[ \sum_{i=1}^{\beta\ell^*(n)} T_i \leq cn \right] =
    e^{-\beta}(1+o(1)). \label{eq:sumT}
  \end{equation}

  For part (ii) we only prove the statement for
  $R_{\beta}(n)=R_{\beta}^1(n)$, the statement for $k>1$ follows easily
  by subtracting the claims with $\beta$ replaced by $\beta k $ and
  $\beta (k-1)$. Consider $\psi_k$ as above, and additionally define
  functions $\varphi_k=\Ind{L(Y(k),n-1)\leq \beta\ell^*(n)}$. Using the
  Markov property and translation invariance,
  \begin{equation}
    \begin{split}
      \mathbb E[|R_{\beta}(n)|] = \sum_{k=0}^{n-1}
      \mathbb E[\psi_k \varphi_k] &= \sum_{k=0}^{n-1} \mathbb E[\psi_k]
      \mathbb P\left[L(0,n-1-k)
        \leq \beta\ell^*(n)\right] \\
      &= \sum_{k=0}^{n-1} \frac{1}{\ell^*(k)} \mathbb P\left[
        \sum_{i=1}^{\beta\ell^*(n)} T_i \geq n-1-k\right].
    \end{split}
    \label{eq:exprn}
  \end{equation}
  If $k\leq (1-\delta)n$ for some $\delta>0$, then
  we can apply \eqref{eq:sumT}. Bounding the probability by one in the
  remaining cases, we see that
  \eqref{eq:exprn} is bounded from above by
  \begin{equation*}
    \begin{split}
      \mathbb E[|R_{\beta}(n)|]
      &\leq \sum_{k=0}^{(1-\delta)n}
      \frac{1}{\ell^*(k)} (1-e^{-\beta})(1+o_{\delta}(1)) +
      \sum_{(1-\delta)n<k<n} \frac{1}{\ell^*(k)}
      \\&= \frac{(1-\delta)n}{\ell^*(n)}
      (1-e^{-\beta})(1+o_{\delta}(1)) + \frac{\delta n}{\ell^*(n)},
    \end{split}
  \end{equation*}
  and from below by
  \begin{equation*}
    \mathbb E[|R_{\beta}(n)|] \geq
    \sum_{k=0}^{(1-\delta)n} \frac{1}{\ell^*(k)}
    (1-e^{-\beta})(1+o_{\delta}(1))= \frac{(1-\delta)n}{\ell^*(n)}
    (1-e^{-\beta})(1+o_{\delta}(1)).
  \end{equation*}
  Sending $\delta\to0$ proves the statement for $\mathbb E[|R_{\beta}(n)|]$.

  To bound the variance, we first note that for $i<i+\delta
  n\leq j \leq (1-\delta)n$, by the Markov property and using
  Assumption~\ref{ass:slowvar} and \eqref{eq:sumT},
  \begin{align}
    \mathbb E\left[\psi_i \varphi_i \psi_j \varphi_j
        \right] &\leq \mathbb E\left[ \psi_i \Ind{L(Y(i),i+\delta n) \leq
        \beta\ell^*(n)} \Ind{Y(k)\neq Y(j),~i+\delta n\leq k < j}
      \Ind{L(Y(j),n)\leq\beta\ell^*(n)} \right] \nonumber \\[2mm]
    &\leq \mathbb E[\psi_i]
    \mathbb P\big[L(0,\delta n) \leq \beta\ell^*(n)\big]
    \mathbb E\big[\psi_{j-i-\delta n}\big]
    \mathbb P\big[L(0,\delta n) \leq
      \beta\ell^*(n)\big] \nonumber \\
    &= \frac{1}{\ell^*(i)} (1-e^{-\beta})
    \frac{1}{\ell^*(j-i-\delta n)}
    (1-e^{-\beta})(1+o_{\delta}(1)). \label{eq:psixi1}
  \end{align}
  The variance of $|R_{\beta}(n)|$ is
  \begin{equation}
    \Var{|R_{\beta}(n)|} = 2\sum_{0\leq i\leq j\leq n-1}
    \mathbb E[\psi_i \varphi_i \psi_j \varphi_j] - \mathbb E[\psi_i
      \varphi_i] \mathbb E[\psi_j
      \varphi_j]. \label{eq:varbeta}
  \end{equation}
  For $i<i+\delta n\leq j \leq (1-\delta)n$ we can use \eqref{eq:psixi1}
  and \eqref{eq:exprn} to get
  \begin{equation}
    \begin{split}
      &\sum_{i<i+\delta n\leq j \leq (1-\delta)n}
      \mathbb E[\psi_i \xi_i \psi_j \xi_j] - \mathbb E[\psi_i \xi_i]
      \mathbb E[\psi_j \xi_j] \\
      &\leq (1-e^{-\beta})^2 \sum_{i<i+\delta n\leq j \leq (1-\delta)n}
      \frac{1}{\ell^*(i)} \left(  \frac{1}{\ell^*(j-i-\delta
            n)}(1+o_{\delta}(1)) - \frac{1}{\ell^*(j)} (1+o_{\delta}(1))
      \right)\\
      &= (1-e^{-\beta})^2 \sum_{i=0}^{(1-2\delta)n} \frac{1}{\ell^*(i)}
      \left( \sum_{j=0}^{(1-2\delta)n-i} \frac{1}{\ell^*(j)}
        (1+o_{\delta}(1)) - \sum_{j=i+\delta n}^{(1-\delta)n}
        \frac{1}{\ell^*(j)} (1+o_{\delta}(1)) \right) \\
      &= o_{\delta}\left( \left( \frac{n}{\ell^*(n)} \right)^2
      \right).
    \end{split}
    \label{eq:varbetagood}
  \end{equation}
  For the remaining $i,j$, using \eqref{eq:exppsipsi} we have
  \begin{equation}
    \begin{split}
      \sum_{i=0}^{n-1}\sum_{\substack{i\leq j <
          i+\delta n \\ (1-\delta)n < j < n}} &\mathbb E[\psi_i \xi_i \psi_j \xi_j] -
      \mathbb E[\psi_i \xi_i] \mathbb E[\psi_j \xi_j] \\
      &\leq \sum_{i=0}^{n-1}\sum_{\substack{i\leq j < i+\delta n \\
          (1-\delta)n < j < n}} \mathbb E[\psi_i \psi_j ] \\
      &\leq \sum_{i=0}^{n-1}\sum_{\substack{i\leq j < i+\delta n \\
          (1-\delta)n < j < n}} \mathbb E[\psi_i] \mathbb E[\psi_{j-i}] \\
      &\leq \sum_{i=0}^{n-1} \frac{1}{\ell^*(i)} \left(
        \sum_{j=i}^{i+\delta n-1} \frac{1}{\ell^*(j)} +
        \sum_{j=(1-\delta)n+1}^{n-1} \frac{1}{\ell^*(j)} \right) \\
      &\leq 2 \frac{\delta
        n^2}{\left(\ell^*(n)\right)^2}(1+o_{\delta}(1)).
    \end{split}
    \label{eq:varbetabad}
  \end{equation}
  Inserting \eqref{eq:varbetagood}, \eqref{eq:varbetabad}
  into \eqref{eq:varbeta} and taking $\delta\to0$ yields
  $\Var{|R_{\beta}(n)|}= o ( (\mathbb E|R_\beta (n)|)^2)$
  and the weak law of large numbers follows.

  Finally, part (iii) is proved in the same way as part (ii). The only difference is that instead of using \eqref{eq:sumT} we note that $L(0,\infty)$ is a geometric random variable with parameter $\gamma$, therefore for every $c>0$,
  \begin{equation*}
    \mathbb P\left[ L(0,cn) = k \right] = \gamma(1-\gamma)^{k-1} (1+o(1)) \text{ as } n\to\infty.
  \end{equation*}
  This completes the proof.
\end{proof}


\renewcommand\MR[1]{\relax\ifhmode\unskip\spacefactor3000
\space\fi \MRhref{#1}{#1}}
\renewcommand{\MRhref}[2]%
{\href{http://www.ams.org/mathscinet-getitem?mr=#1}{MR #2}}
\providecommand{\href}[2]{#2}

\bibliographystyle{amsalpha}
\bibliography{references}

\providecommand{\bysame}{\leavevmode\hbox to3em{\hrulefill}\thinspace}
\providecommand{\MR}{\relax\ifhmode\unskip\space\fi MR }
\providecommand{\MRhref}[2]{%
  \href{http://www.ams.org/mathscinet-getitem?mr=#1}{#2}
}
\providecommand{\href}[2]{#2}
\begin{thebibliography}{BC{\v C}R14}

\bibitem[B{\v{C}}07]{BC07}
G.~Be{n A}rous and J.~{\v{C}}ern{\'y}, \emph{Scaling limit for trap models on
  {$\ZZ^d$}}, Ann. Probab. \textbf{35} (2007), no.~6, 2356--2384. \MR{2353391}

\bibitem[B{\v{C}}11]{BC11}
M.~T. Barlow and J.~{\v{C}}ern{\'y}, \emph{Convergence to fractional kinetics
  for random walks associated with unbounded conductances}, Probab. Theory
  Related Fields \textbf{149} (2011), no.~3-4, 639--673. \MR{2776627}

\bibitem[BC{\v C}R14]{BCCR13}
G.~{Be{n A}rous}, M.~{Cabezas}, J.~{{\v C}ern{\'y}}, and R.~{Royfman},
  \emph{{Randomly Trapped Random Walks}}, To appear in Ann. Probab. (2014).

\bibitem[B{\v{C}}M06]{BCM06}
G.~Be{n A}rous, J.~{\v{C}}ern{\'y}, and T.~Mountford, \emph{Aging in
  two-dimensional {B}ouchaud's model}, Probab. Theory Related Fields
  \textbf{134} (2006), no.~1, 1--43. \MR{2221784}

\bibitem[BS02]{BS02}
E.~Bolthausen and A.-S. Sznitman, \emph{On the static and dynamic points of
  view for certain random walks in random environment}, Methods Appl. Anal.
  \textbf{9} (2002), no.~3, 345--375, Special issue dedicated to Daniel W.
  Stroock and Srinivasa S. R. Varadhan on the occasion of their 60th birthday.
  \MR{2023130}

\bibitem[{\v{C}}er07]{C07}
J.~{\v{C}}ern{\'y}, \emph{Moments and distribution of the local time of a
  two-dimensional random walk}, Stochastic Process. Appl. \textbf{117} (2007),
  no.~2, 262--270. \MR{2290196}

\bibitem[Dar52]{Dar52}
D.~A. Darling, \emph{The influence of the maximum term in the addition of
  independent random variables}, Trans. Amer. Math. Soc. \textbf{73} (1952),
  95--107. \MR{0048726}

\bibitem[DE51]{DE51}
A.~Dvoretzky and P.~Erd{\H{o}}s, \emph{Some problems on random walk in space},
  Proceedings of the {S}econd {B}erkeley {S}ymposium on {M}athematical
  {S}tatistics and {P}robability, 1950, University of California Press,
  Berkeley and Los Angeles, 1951, pp.~353--367. \MR{0047272}

\bibitem[ET60]{ET60}
P.~Erd{\H{o}}s and S.~J. Taylor, \emph{Some problems concerning the structure
  of random walk paths}, Acta Math. Acad. Sci. Hungar. \textbf{11} (1960),
  137--162. (unbound insert). \MR{0121870}

\bibitem[FM13]{FM}
L.~R.~G. Fontes and P.~Mathieu, \emph{On the dynamics of trap models in
  $\mathbb{Z}^d$}, Proceedings of the London Mathematical Society (2013).

\bibitem[GS13]{GS13}
V.~{Gayrard} and A.~{Svejda}, \emph{{Convergence of clock processes on infinite
  graphs and aging in Bouchaud's asymmetric trap model on $\ZZ^d$}},
  \href{http://arxiv.org/abs/math/1309.3066}{arXiv:math/1309.3066} (2013).

\bibitem[JS03]{JS03}
J.~Jacod and A.~N. Shiryaev, \emph{Limit theorems for stochastic processes},
  second ed., Grundlehren der Mathematischen Wissenschaften [Fundamental
  Principles of Mathematical Sciences], vol. 288, Springer-Verlag, Berlin,
  2003. \MR{1943877}

\bibitem[Kal02]{K02}
O.~Kallenberg, \emph{Foundations of modern probability}, second ed.,
  Probability and its Applications (New York), Springer-Verlag, New York, 2002.
  \MR{1876169}

\bibitem[KS63]{KS63}
H.~Kesten and F.~Spitzer, \emph{Ratio theorems for random walks. {I}}, J.
  Analyse Math. \textbf{11} (1963), 285--322. \MR{0162279}

\bibitem[Mou11]{M11}
J.-C. Mourrat, \emph{Scaling limit of the random walk among random traps on
  {$\ZZ^d$}}, Ann. Inst. Henri Poincar\'e Probab. Stat. \textbf{47} (2011),
  no.~3, 813--849. \MR{2841076}

\bibitem[MS04]{MS04}
M.~M. Meerschaert and H.-P. Scheffler, \emph{Limit theorems for continuous-time
  random walks with infinite mean waiting times}, J. Appl. Probab. \textbf{41}
  (2004), no.~3, 623--638. \MR{2074812}

\bibitem[MW65]{MW65}
E.~W. Montroll and G.~H. Weiss, \emph{Random walks on lattices. {II}}, J.
  Mathematical Phys. \textbf{6} (1965), 167--181. \MR{0172344}

\bibitem[Pet83]{P83}
K.~Petersen, \emph{Ergodic theory}, Cambridge Studies in Advanced Mathematics,
  vol.~2, Cambridge University Press, Cambridge, 1983. \MR{833286}

\bibitem[Pit74]{P74}
J.~H. Pitt, \emph{Multiple points of transient random walks}, Proc. Amer. Math.
  Soc. \textbf{43} (1974), 195--199. \MR{0386021}

\bibitem[R{\'e}v13]{Rev}
P.~R{\'e}v{\'e}sz, \emph{Random walk in random and non-random environments},
  third ed., World Scientific Publishing Co. Pte. Ltd., Hackensack, NJ, 2013.
  \MR{3060348}

\bibitem[Sen76]{Sen76}
E.~Seneta, \emph{Regularly varying functions}, Lecture Notes in Mathematics,
  no. Nr. 508, Springer-Verlag, 1976.

\bibitem[Spi76]{S76}
F.~Spitzer, \emph{Principles of random walk}, second ed., Springer-Verlag, New
  York-Heidelberg, 1976, Graduate Texts in Mathematics, Vol. 34. \MR{0388547}

\bibitem[Whi02]{W02}
W.~Whitt, \emph{Stochastic-process limits}, Springer Series in Operations
  Research, Springer-Verlag, New York, 2002, An introduction to
  stochastic-process limits and their application to queues. \MR{1876437}

\end{thebibliography}

\end{document}